\documentclass[12pt]{article}
\usepackage[margin = 20mm]{geometry}
\usepackage{amsmath}
\usepackage{amsfonts}
\usepackage{mathrsfs}
\usepackage{amssymb}
\usepackage{amsthm}
\usepackage{cases}
\usepackage{blkarray}
\usepackage{enumerate}
\usepackage{tikz}
\usepackage{pgfplots}
\pgfplotsset{compat=1.5}
\usepackage{tocloft}
\usepackage{tabularx}
\usepackage{setspace}
\usepackage{array,color,colortbl}   
\providecommand{\mk}{\cellcolor{blue!15}}
\usepackage{abstract}
\usepackage{nicematrix}
\usepackage{xcolor, colortbl}
\usepackage[ruled]{algorithm2e}
\usepackage[backend = biber, doi = false, url = false, isbn = false, maxbibnames=222]{biblatex}
\renewbibmacro{in:}{}
\newtheorem{thm}{Theorem}[section]

\newtheorem{lem}[thm]{Lemma}

\theoremstyle{definition}

\numberwithin{equation}{section}
\def\eref#1{$(\ref{#1})$}
\def\sref#1{\S$\ref{#1}$}
\def\lref#1{Lemma~$\ref{#1}$}
\def\tref#1{Theorem~$\ref{#1}$}

\def\X{\mathscr{X}}
\def\inc{\hookrightarrow}
\renewcommand{\geq}{\geqslant}
\renewcommand{\leq}{\leqslant}
\renewcommand{\ge}{\geqslant}
\renewcommand{\le}{\leqslant}
\renewcommand{\emptyset}{\varnothing}
\makeatletter
\g@addto@macro\bfseries{\boldmath}
\makeatother

\title{Latin squares without proper subsquares}
\author{Jack Allsop \ \ Ian M. Wanless\\
	\small School of Mathematics\\[-0.5ex]
	\small Monash University\\[-0.5ex]
	\small Vic 3800, Australia\\
	\small\tt jack.allsop@monash.edu \ \ ian.wanless@monash.edu}

\date{}

\begin{document}
	
	\maketitle
	
	\begin{abstract}
		A $d$-dimensional Latin hypercube of order $n$ is a $d$-dimensional
		array containing symbols from a set of cardinality $n$ with the
		property that every axis-parallel line contains all $n$ symbols
		exactly once. We show that for $(n, d) \notin \{(4,2), (6,2)\}$
		with $d \geq 2$ there exists a $d$-dimensional Latin hypercube of
		order $n$ that contains no $d$-dimensional Latin subhypercube of any
		order in $\{2,\dots,n-1\}$.  The $d=2$ case settles a 50 year old
		conjecture by Hilton on the existence of Latin squares without
		proper subsquares.
	\end{abstract}
	
	\section{Introduction}
	
	Let $n$ be positive integer. A \emph{Latin square} of order $n$ is an
	$n \times n$ matrix of $n$ symbols, such that each symbol occurs
	exactly once in each row and column. Let $L$ be a Latin square of
	order $n$. A \emph{subsquare} of order $k$ in $L$ is a $k \times k$
	submatrix of $L$ that is itself a Latin square.
        Clearly $L$ has $n^2$ subsquares of order one, and
	one subsquare of order $n$. A subsquare of $L$ of order $k \in
	\{2,3,\ldots, n-1\}$ is called \emph{proper}. A subsquare of
	order two is called an \emph{intercalate}. If $L$ has no
	proper subsquares then it is called $N_\infty$. Thousands of
	papers and several books have been written about properties
	and applications of Latin squares \cite{ls,DKII,Eva18,LM98}.
	However, one of the most natural and prominent questions in
	the area has defied solution until now.  Hilton conjectured
	that an $N_\infty$ Latin square of order $n$ exists for all
	sufficiently large $n$ (his conjecture was first stated
	in~\cite{ls}, albeit incorrectly).  From a series of papers
	including \cite{amninf,heinninf,ninfodd}, it has long been
	known that an $N_\infty$ Latin square of order $n$ exists for
	all positive integers $n$ not of the form $2^x3^y$ for
	integers $x \geq 1$ and $y \geq 0$. It had also been shown
	\cite{1subsq,WanThesis} that Hilton's conjecture holds for all
	orders $n<256$. In this work, we resolve his conjecture by
	constructing $N_\infty$ Latin squares for all previously
        unresolved orders.
	
Latin squares are part of a more general family of combinatorial objects called Latin hypercubes. For a positive integer $m$ let $[m] = \{1,2,\ldots, m\}$. Let $n$ and $d$ be positive integers. For each $i \in [d]$ let $I_i$ be a set of cardinality $n$, and let $I = I_1 \times I_2 \times \cdots \times I_d$. Consider $H : I \to \Sigma$ for some set $\Sigma$ of cardinality $n$. Denote the image of $(x_1,x_2,\ldots, x_d) \in I$ under $H$ by $H[x_1,x_2,\ldots, x_d] \in \Sigma$. We can naturally think of $H$ as a $d$-dimensional array whose $i$-th axis is indexed by $I_i$. We will treat the map $H$ and the corresponding array as interchangeable objects.
The array $H$ is a $d$-dimensional \emph{Latin hypercube} of order $n$ if
\[
\{H[c_1,\ldots, c_{k-1}, x, c_{k+1}, \ldots, c_d] : x \in I_k\} = \Sigma,
\]
for each $c = (c_1,c_2,\ldots, c_d) \in I$ and $k \in [d]$.
A one-dimensional Latin hypercube is a permutation and a two-dimensional Latin hypercube is a Latin square. A \emph{Latin cube} is a three-dimensional Latin hypercube. Suppose that $d \geq 2$ and let $k \leq n$ be an integer. For each $i \in [d]$ let $S_i \subseteq I_i$ be of cardinality $k$. The restriction $H' = H|_{S_1 \times S_2 \times \cdots \times S_d}$ of $H$ is a \emph{subarray} of $H$. If $H'$ contains exactly $k$ symbols then $H'$ is called a \emph{subhypercube} of $H$.
If $k \in \{2,3,\ldots, n-1\}$ then $H'$ is called \emph{proper}. If a Latin hypercube contains no proper subhypercubes then it is called $N_\infty$. When dealing with hypercubes $H : I_1 \times \cdots \times I_d \to \Sigma$ we will generally assume that each $I_i = [n]$ and also that $\Sigma = [n]$. However, we need to allow subhypercubes to have more general index sets.

The goal of this paper is to resolve the existence problem for
$N_\infty$ Latin hypercubes by proving the following theorem.

\begin{thm}\label{t:ninfhyper}
	Let $d \geq 2$ and $n$ be positive integers. There exists an
	$N_\infty$ Latin hypercube of order $n$ and dimension $d$ if and
	only if $(n, d) \not\in \{(4,2), (6,2)\}$.
\end{thm}

The $d=2$ case of \tref{t:ninfhyper} resolves Hilton's conjecture.

The structure of this paper is as follows. In \sref{s:back} we present
some background material and motivation. In
\sref{s:main} we resolve the existence problem of $N_\infty$ Latin
squares by constructing $N_\infty$ Latin squares of orders of the form
$2^x3^y$.  In \sref{s:main2} we extend our results from \sref{s:main}
to prove \tref{t:ninfhyper}. Finally, in \sref{s:conc} we give some
brief concluding remarks.

\section{Background}\label{s:back}

In this section we motivate the study of $N_\infty$ Latin squares. We also introduce some material needed to prove \tref{t:ninfhyper}.

A Latin square is called $N_2$ if it contains no intercalates. It is
known~\cite{DKII,n2conj, kotzturg, mcleishn2,mcleishn2corr} that an $N_2$
Latin square of order $n$ exists if and only if $n \not\in \{2,
4\}$. One of the early motivations to study $N_2$ Latin squares was a
connection with disjoint Steiner triple systems~\cite{n2conj}. More
recently, $N_2$ Latin squares have been shown to be very
rare~\cite{KSS22,KSSS22,KS18,lotssub}. No estimates have been
proved for the proportion of Latin squares that are
$N_\infty$. However, $N_\infty$ is a strictly stronger property than
$N_2$, and hence very rarely achieved among Latin squares. This,
together with the fact that most direct and recursive construction
techniques inherently create subsquares, accounts for why Hilton's
conjecture has defied solution until this point.

Another reason to study $N_\infty$ Latin squares is due to their
connection with so called perfect $1$-factorisations of graphs. A
\emph{$1$-factor} of a graph $G$ is a subset $H$ of the edges of $G$
so that every vertex of $G$ is incident to exactly one edge in $H$. A
\emph{$1$-factorisation} of $G$ is a partition of the edges of $G$
into $1$-factors. Any pair of distinct $1$-factors in a
$1$-factorisation $F$ induces a $2$-regular subgraph of $G$. If this
subgraph is a Hamiltonian cycle in $G$, regardless of the pair of
$1$-factors, then $F$ is called a \emph{perfect
	$1$-factorisation}. Much work has been done on constructing perfect
$1$-factorisations of complete graphs and complete bipartite
graphs. Let $n$ be an odd integer. An $N_\infty$ Latin square of order
$n$ can be constructed from a perfect $1$-factorisation of the
complete graph $K_{n+1}$, or from a perfect $1$-factorisation of the
complete bipartite graph $K_{n, n}$.
It is not necessarily true that an $N_\infty$ Latin square of order
$n$ implies the existence of a perfect $1$-factorisation of $K_{n, n}$
or $K_{n+1}$. Indeed, the Latin squares built from perfect
$1$-factorisations have an even stronger property than $N_\infty$;
namely they do not contain Latin rectangles other than those
consisting of entire rows of the Latin square.  For further details,
see \cite{WI05}.

We now present some material regarding Latin squares that we require
in order to prove \tref{t:ninfhyper}. Unless otherwise stated, the
rows and columns of a matrix of order $n$ will be indexed by $[n]$,
and the symbol set will be $[n]$. When dealing with the set $[n]$, all
calculations will be modulo $n$.  Let $L$ be a matrix of order $n$. We
can think of $L$ as a set of $n^2$ triples of the form $(\text{row},
\text{column}, \text{symbol})$. We will sometimes use set notation for
matrices, e.g. if $L$ contains the triple $(1,1,1)$ then we will
write $(1,1,1) \in L$. Each triple of $L$ is called an
\emph{entry}. The entry $(i, j, k)$ occurs in \emph{cell} $(i, j)$ of
$L$. We also write $L[i, j] = k$. The \emph{principal entry} of $L$ is
the entry in cell $(1,1)$. Let $M$ be another matrix (not necessarily
of order $n$), and let $S$ be a set of entries of $M$. Suppose that
each entry in $S$ is in a cell $(i, j)$ for some $\{i, j\} \subseteq[n]$.
Then the \emph{shadow of $S$ in $L$} is the set of entries
$\{(i, j, L[i, j]) : (i, j, M[i, j]) \in S\}$.

Let $L$ be a Latin square. Any Latin square that can be obtained from $L$ by permuting its rows, permuting its columns and renaming its symbols is said to be \emph{isotopic} to $L$. Any Latin square that can be obtained from $L$ by uniformly permuting the coordinates of each entry of $L$ is said to be a \emph{conjugate} of $L$. Each Latin square has six (not necessarily distinct) conjugates. The \emph{species} of $L$ is the set of Latin squares that are isotopic to some conjugate of $L$. The $N_\infty$ property is a species invariant. The concept of a species generalises naturally to Latin hypercubes.

Let $L$ be a Latin square of order $n$, let $\{i, j\} \subseteq [n]$
and let $\sigma$ be any symbol other than $L[i, j]$. The matrix obtained from
$L$ by replacing the entry $(i, j, L[i, j])$ by $(i, j, \sigma)$ is
denoted by $\sigma \inc L[i, j]$. Such a matrix is called a \emph{near
	copy} of $L$. We stress that $\sigma$ may or may not be a member of
$[n]$, but if it is then the near copy will contain two copies of
$\sigma$ in row $i$ and column $j$.  More generally, let $k \leq n^2$
and let $\{(x_i,y_i) : i \in [k]\}$ be a set of $k$ distinct cells in
$L$. Also let $\sigma_i$ be a symbol other than $L[x_i, y_i]$ for each
$i \in [k]$. Let $L'$ be obtained from $L$ by replacing each entry
$(x_i, y_i, L[x_i, y_i])$ by $(x_i,y_i,\sigma_i)$. Then $L'$ is called
a \emph{$k$-near copy} of $L$. In particular, we can say that $L$ is a
$0$-near copy of $L$, and a $1$-near copy of $L$ is simply a near copy
of $L$. When considering $k$-near copies of Latin squares we will
still use Latin square terminology such as subsquares.  The entries
$\{(x_i,y_i,\sigma_i):i\in[k]\}$ are called the \emph{alien entries}
of $L'$ with respect to $L$. If there is no ambiguity as to the matrix
$L$ then we will simply call these entries the alien entries of
$L'$. If $T$ is a submatrix of $L'$ that contains an alien entry
$\pi$ of $L'$ with respect to $L$ then we will say that $\pi$ is an
alien entry of $T$ with respect to $L$. We adopt the same convention
for the following definitions. All entries of $L'$ that are not alien
entries are called the \emph{native entries} of $L'$. Every symbol in
a native entry of $L'$ is called a \emph{native symbol} of $L'$. The
cells $\{(x_i, y_i) : i \in [k]\}$ are called the \emph{holes} in
$L'$. The symbol $L[x_i, y_i]$ is called the the \emph{displaced
	native} from the hole $(x_i,y_i)$.

It is well known that any proper subsquare of a Latin square $L$
cannot be bigger than half the order of $L$. However, this is not true
in near copies of Latin squares, as demonstrated by the shaded subsquare in
\[
\left[\begin{array}{ccc}
	\mk1&\mk2&3\\
	\mk2&\mk1&1\\
	3&1&2\\
\end{array}
\right].
\]
Nevertheless, subsquares in near copies cannot be much bigger than half the
order of the parent Latin square.

\begin{lem}\label{l:nolargesubsq}
	Let $M$ be a near copy of a Latin square $L$ of order $n>1$. Suppose
	that $S$ is a subsquare of $M$ of order $s$. Then $s\le(n+1)/2$.
\end{lem}

\begin{proof}
	Let $T$ be the submatrix of $M$ induced by the rows and columns that
	do not hit $S$. Note that $T$ is not empty, because $M$ is not a
	Latin square.  Suppose that $\sigma$ is any symbol that occurs in
	$S$ and does not occur in the alien entry in $M$ with respect to
	$L$. Then $\sigma$ must occur $s$ times in $S$ and $n-s$ times in
	$T$. We have at least $s-1$ choices for $\sigma$, but we only have
	room for $n-s$ of them within $T$.  Hence $s-1\le n-s$, as required.
\end{proof}

Another simple result we will need is the following. It is a
restatement of a well known result on the smallest Latin trade.

\begin{lem}\label{l:mintrade}
  Let $L$ and $M$ be distinct Latin squares on the same set of
  symbols.  If $M$ is a $k$-near copy of $L$ then $k\ge4$, with
  equality only possible if $L$ and $M$ both contain an intercalate.
\end{lem}

Let $L$ be a Latin square of order $n$ and let $M$ be a Latin square of order $m$. The \emph{direct product} of $L$ and $M$, denoted by $L \times M$, is a Latin square of order $nm$ whose rows and columns are indexed by $[n] \times [m]$. It is defined by $(L \times M)[(i, j), (x, y)] = (L[i, x], M[j, y])$.
There are two natural ways of ordering the rows and columns of $L \times M$. The first way is to use the order $\prec_1$ on $[n] \times [m]$, where we order by the first coordinate and use the second coordinate to break ties. When ordering in this way, $L \times M$ decomposes into $n^2$ blocks, each of which is isotopic to $M$. These are known as \emph{$M$-blocks}. The second way is to use the order $\prec_2$ on $[n] \times [m]$, where we order by the second coordinate and use the first to break ties. 
With this ordering the square $L \times M$ decomposes into $m^2$ blocks, each of which is isotopic to $L$. These are known as \emph{$L$-blocks}. We will refer to $M$-blocks and $L$-blocks collectively as \emph{blocks}. Let $\Phi_1 : [n] \times [m] \to [n]$ denote the projection onto the first coordinate, and let $\Phi_2 : [n] \times [m] \to [m]$ denote the projection onto the second coordinate. Let $S$ be a submatrix of $L \times M$. The \emph{projection of $S$ onto $L$}, denoted by $\Phi_1(S)$, is the set of triples $\{(\Phi_1(r), \Phi_1(c), \Phi_1(s)) : (r, c, s) \in S\}$. Similarly, the \emph{projection of $S$ onto $M$}, denoted by $\Phi_2(S)$, is the set of triples $\{(\Phi_2(r), \Phi_2(c), \Phi_2(s)) : (r, c, s) \in S\}$. The projections of $S$ onto the first and second coordinates can be defined on any matrix whose row indices, column indices and symbols are $[n] \times [m]$.
Let $T$ be an $M$-block of $L \times M$. Then $\Phi_1(T)$ consists of a single entry of $L$, say $(i, j, L[i, j])$. The \emph{position} of $T$ is then defined to be $(i, j)$. The position of an $L$-block is defined similarly using $\Phi_2$. The $M$-block in position $(1,1)$ is called the \emph{principal $M$-block} of $L \times M$ and the $L$-block in position $(1,1)$ is called the \emph{principal $L$-block} of $L \times M$. The principal entry of the $M$-block in position $(i, j)$ is the entry of $L \times M$ in cell $((i, 1), (j, 1))$. The principal entry of the $L$-block in position $(i, j)$ is the entry of $L \times M$ in cell $((1,i), (1,j))$.

In order to prove \tref{t:ninfhyper} we need the so called \emph{corrupted product} defined in~\cite{1subsq}. Let $A$ be an $N_\infty$ square of order $\alpha$ and let $B$ be a square isotopic to $A$, with the same symbol set as $A$. The pair $(A, B)$ is a \emph{corrupting pair} of order $\alpha$ if:
\begin{itemize}
	\item $A[i, j] = B[i, j]$ if and only if $i=j=1$, and,
	\item for all $\{i, j\} \subseteq [\alpha]$, there is no proper subsquare of $B[i, j] \inc A[i, j]$ involving the principal entry.
\end{itemize}

Let $(A, B)$ be a corrupting pair of order $\alpha$, let $M$ be an
$N_\infty$ square of order $\mu$ and let $s \in [\mu-1]$. The
\emph{corrupted product} $P = (A, B) *_s M$ of shift $s$, whose rows
and columns are indexed by $[\alpha] \times [\mu]$, is defined by,
\[
P[(i, j), (k, l)] = \begin{cases}
	(A[i, k], M[j, l]+s) & \text{if } i=k=1,\\
	(B[i, k], M[j, l]) & \text{if } j=l=1 \text{ and } (i, k) \neq (1,1), \\
	(A[i, k], M[j, l]) & \text{otherwise}.
\end{cases}
\] 
We can obtain $P$ from the direct product $A \times M$ as follows. First, replace the principal $A$-block of $A \times M$ by the principal $B$-block of $B \times M$. Then add $s$ to the $M$-coordinate of each symbol in the principal $M$-block. See~\cite{1subsq} for a more detailed description of corrupted products.
When discussing corrupted products and other Latin squares that can be obtained from direct products by a small number of perturbations, we will use the terminology such as blocks, positions and projections that we introduced for direct products. So each $A$-block of $P$ is a near copy of a Latin square that is isotopic to $A$. The principal $M$-block of $P$ is a subsquare of $P$, which we will denote by $\beta_M$. Any other $M$-block of $P$ is a near copy of a Latin square that is isotopic to $M$.
If the matrix $M[i, j]+s \inc M[i, j]$ does not contain a subsquare isotopic to $A$ for any $\{i, j\} \subseteq [\mu]$, then $s$ is called an \emph{allowable shift} with respect to $(A, M)$. If $A$ is of order $\alpha$ and $M[i, j]+s \inc M[i, j]$ does not contain a subsquare of order $\alpha$ for any $\{i, j\} \subseteq [\mu]$ then $s$ is called a \emph{strong allowable shift} with respect to $(A, M)$. We can now state the following result from \cite{1subsq}, which is our motivation for discussing corrupted products.

\begin{thm}\label{t:1sub}
	Let $(A, B)$ be a corrupting pair of order $\alpha$, let $M$ be an
	$N_\infty$ square of order $\mu>\alpha$ and let $s \in [\mu-1]$. If
	$s$ is an allowable shift with respect to $(A, M)$, then the only
	proper subsquare of the corrupted product $(A, B) *_s M$ is $\beta_M$.
\end{thm}

\medskip

Let $L$ be a Latin square of order $n$. For each $\{i,j\}\subseteq[n]$
with $i \neq j$, the permutation mapping row $i$ to row $j$, denoted
by $\tau_{i, j}$, is defined by $\tau_{i, j}(L[i, k]) = L[j, k]$ for all
$k\in[n]$. Such permutations are called \emph{row permutations} of
$L$. Let $\rho$ be a cycle in $\tau_{i, j}$  and in row $i$ (or row $j$)
let the set of columns containing the symbols involved in $\rho$ be $C$.
The set of entries
in cells $\{i, j\} \times C$ is called a
\emph{row cycle} of $L$. The \emph{length} of this row cycle is $|C|$.  Denote
by $\rho(i,j,c)$ the row cycle induced by the cycle in $\tau_{i, j}$ that
hits column $c$. Column cycles and symbol cycles can be
defined similarly to row cycles.  These cycles can be used to create new
Latin squares from old ones, in a method known as cycle
switching \cite{cycswitch}.
Suppose that there is a row cycle $\rho(i, j, c)$ of $L$,
and let $C$ be the set of columns hit by this row cycle. A new
Latin square $L'$ can be defined by
\[
L'[x, y] = \begin{cases}
	L[i, y] & \text{if } x=j \text{ and } y \in C, \\
	L[j, y] & \text{if } x=i \text{ and } y \in C, \\
	L[x, y] & \text{otherwise}.
\end{cases}
\]
We will say that $L'$ has been obtained from $L$ by \emph{switching}
on the cycle $\rho(i, j, c)$.

To prove \tref{t:ninfhyper} we will first resolve Hilton's conjecture by constructing an $N_\infty$ Latin square of order $n$ for any $n \not\in \{4,6\}$ of the form $2^x3^y$ with $x \geq 1$ and $y \geq 0$. The construction is recursive and will work as follows. Given an $N_\infty$ Latin square of order $\mu$, we use corrupted products to construct Latin squares of order $8\mu$ and $9\mu$ that contain exactly one proper subsquare. We then use cycle switching to destroy this subsquare, in such a way as to not create any new subsquares.

We now describe another trade which can be used to create new Latin squares from old ones, which is similar to cycle switching. This method will only be used to construct $N_\infty$ squares that we need as the base cases for our recursive construction. 
Let $L$ be a Latin square of order $n$. Suppose that there are three distinct rows $i, j$ and $k$, distinct columns $x$ and $y$, and symbols $a$ and $b$ of $L$ such that: $L[i, x] = a = L[k, y]$, $L[i, y] = b = L[j, x]$ and $b$ is contained in the cycle of the row permutation $\tau_{j, k}$ of $L$ that contains $a$. Write this cycle as $(a, z_1,z_2,\ldots, z_\ell, b, \ldots)$. Let $c_0 \in [n]$ be such that $L[j, c_0] = a$ and for $w \in [\ell]$ let $c_w \in [n]$ be such that $L[j, c_w] = z_w$. A Latin square $L'$ can be defined by,
\[
L'[u, v] = \begin{cases}
	b & \text{if } (u, v) \in \{(i, x), (k, y), (j, c_\ell)\}, \\
	a & \text{if } (u, v) \in \{(i, y), (j, x), (k, c_0)\}, \\
	z_w & \text{if } (u, v) \in \{(j, c_{w-1}), (k, c_w)\}, w \in [\ell], \\
	L[u, v] & \text{otherwise}.
\end{cases}
\]
We will let $\eta(i, j, x)$ denote the set of entries in cells $$\{(i, x), (i, y), (j, x), (k, y)\} \cup \{(j, c_w), (k, c_w) : w \in [\ell] \cup \{0\}\},$$ 
and we will say that $L'$ has been obtained from $L$ by switching on $\eta(i, j, x)$.

Consider \eref{e:E} below. The highlighted symbols of this Latin square form $\eta(4,7,2)$. Switching on $\eta(4,7,2)$ involves swapping each highlighted symbol with the other highlighted symbol in the same column.
\begin{align}\label{e:E}
	\setlength{\arraycolsep}{0pt}
	\begin{array}{rc}
	  &
	  \setlength{\arraycolsep}{6pt}  
		\begin{array}{ccccccccc}
			\!\!c_0\,\,&\!\!x\,\,&&\,y\!&c_1&\phantom{0}&\phantom{0}&\phantom{0}\\
		\end{array}
		\\[0ex]
		\begin{array}{c}
			\\
			k\\
			\\
			i\\
			\\
			\\
			j\\
			\\
		\end{array}
		&
		\setlength{\arraycolsep}{6pt}    
		\left[\begin{array}{ccccccccc}
		1&2&3&4&5&6&7&8\\
		\mk2&3&5&\mk7&\mk8&1&6&4\\
		3&1&8&5&6&4&2&7\\
		4&\mk7&1&\mk8&3&2&5&6\\
		5&6&7&1&4&3&8&2\\
		6&4&2&3&7&8&1&5\\
		\mk7&\mk8&4&6&\mk2&5&3&1\\
		8&5&6&2&1&7&4&3\\
	\end{array}\right].
	\\
\end{array}
\end{align}

\section{Latin squares without subsquares}\label{s:main}

In this section we resolve Hilton's conjecture by proving the following theorem.

\begin{thm}\label{t:ninf}
There exists an $N_\infty$ Latin square of order $n$ for all
$n\not\in\{4,6\}$.
\end{thm} 

To do this we will need some preliminary lemmas. Let $L$ be an $n \times n$ matrix and let $R$ and $C$ be subsets of $[n]$. The submatrix of $L$ induced by the rows in $R$ and the columns in $C$ is denoted by $L[R, C]$. 
We will index the rows and columns of $L[R, C]$ by $R$ and $C$.

\begin{lem}\label{l:nearrow} 
Let $L$ be a near copy of a Latin square $L'$ with associated alien
entry $\pi$. Let $T$ be a submatrix of $L$ that does not contain
$\pi$. Suppose that $T$ is a $k$-near copy of some $N_\infty$ square
$N$, where $k\in\{0,1,2\}$. 
Suppose further that no symbol of an entry of $T$ that is alien with
respect to $N$ is native to $T$.
Also suppose that $L$ has a subsquare $S$ that meets $T$ in
at least two entries. Let $V = S\cap T$ and suppose that
\begin{itemize}
	\item $V$ has more than $k$ columns, and
	\item $V$ intersects a row $r$ of $L$ that contains none of the holes
	in $T$ with respect to $N$.
\end{itemize}
Then one of the following is true: 
\begin{itemize}
\item $V=T$,
\item $V$ has exactly $k$ rows and $k+1$ columns,
\item $k=2$,
  and the two alien entries of $T$ with respect to $N$ are
  $(x_1,y_1,\sigma_1) \in V$ and $(x_2,y_2,\sigma_2) \not\in V$. The shadow
  of $V$ is a subsquare of the matrix $\nu\inc N[x_1,y_1]$ where $\nu$
  is the displaced native from $(x_2,y_2)$
  Also, either $x_1=x_2$ and $\pi$ is in column $y_1$, or
  $y_1=y_2$ and $\pi$ is in row $x_1$.
	
\end{itemize}
\end{lem}

\begin{proof}
Let $R$ be the set of rows of $V$ and let $C$ be the set of columns of $V$.
We will assume that $|R| \neq k$ or $|C| \neq k+1$, since otherwise the Lemma
holds.
Throughout this proof whenever we use the terms displaced native,
native symbol and hole, they will be with respect to $N$. Let $\Sigma$
be the set of native symbols of $T$ in $V$. Since $|C|>k$ there is
some column $c\in C$ that contains no holes in $T$.

Aiming for a contradiction, suppose that $|R| > |C|$. Since $|C|>k$
it follows that $|R| \geq k+2$ and $V$ must intersect at least two
rows that contain none of the holes in $T$. Let $r'$ be one such row
that does not contain $\pi$.  Each of the $|R|$ symbols in column
$c$ of $V$ is native to $T$, and hence must occur in row $r'$ in
$T$. But these symbols must also occur in row $r'$ of $S$, since $S$
is a Latin square. So there are at least $|R|$ symbols in row $r'$ in $V$,
which forces $|C|\geq|R|$.

We now show that $|R| = |C| = |\Sigma|$. Row $r$ of $T$ contains
only native symbols of $T$ and so $|\Sigma| \geq |C|$. Now suppose,
for a contradiction, that $|\Sigma| > |R|$ and hence there is some
symbol $\sigma \in \Sigma$ that does not occur in column $c$ of
$V$. But $\sigma$ does occur in column $c$ of $T$, say in row
$r''$. Since $S$ is a Latin square that contains column $c$ and
symbol $\sigma$ but not row $r''$, it follows that $\pi$ must occur
in column $c$ and have symbol $\sigma$. As $L$ contains only one
alien entry with respect to $L'$ we know that our choices for $c$
and $\sigma$ were both forced. It follows that
$|R|+1 = |\Sigma| \geq |C| = k+1\geq |R|$.
As we are assuming that $|R|\ne k$ or $|C|\ne k+1$, the only remaining
possibility is that $|R|=|C|=k+1$ and $|\Sigma|=k+2$.
Since $|\Sigma| > |C|$ it follows that there is a symbol
$\sigma'\in\Sigma$ that does not occur in row $r$ of $V$.
Transposing the argument we just used for $\sigma$, but applying it
to $\sigma'$, we deduce that $\pi$ must occur in row $r$.
But we know that $\pi$ occurs in column $c$, and so $\pi$ cannot
occur in row $r$ because $\pi \not\in T$. This contradiction implies
that $|R| = |C| = |\Sigma|$.

Consider the $|R|\times|R|$ submatrix $M$ of $N$ that is the shadow of
$V$. If $M$ contains exactly $|R|$ symbols then it is a subsquare of
$N$, so must be equal to $N$, as $N$ is $N_\infty$. In that case we
would have $V=T$, so we may assume that $V$ contains a hole $(x,y)$
such that $M[x,y]\notin\Sigma$.  Row $r$ contains no hole in $T$ and
thus there is a symbol $\theta\in\Sigma$ that occurs in row $r$ of $M$
but not in row $x$ of $M$. Since $S$ is a Latin square that contains
row $x$ and symbol $\theta$, we must have that either (i)~$\pi$ occurs
in row $x$ and contains symbol $\theta$ or (ii)~$\theta$ is the
displaced native from a hole $(x,y')$ in row $x$.  In the latter case,
$(x,y')$ is outside of $V$ because $\theta$ does not occur in row $x$
of $M$. Similar logic can be applied to show that
an analogue of options (i) or (ii) must also hold for columns.
However, $\pi$ cannot be in row
$x$ and also in column $y$ because $\pi\notin T$. Also there is at
most one hole other than $(x,y)$. We conclude that there must be two
holes, with the second hole lying in whichever of row $x$ and column
$y$ does not contain $\pi$. Let $\nu$ be the displaced native from
the hole that is not $(x,y)$ and consider the matrix $M'=\nu\inc M[x,y]$.
We note that the symbols in $M'$ must be precisely $\Sigma$, and that
no symbol is duplicated within any row or column of $M'$ with 
the possible exception that $\nu$ might occur
twice within row $x$ or within column $y$ (but not both).
A consequence is that each of the $|R|$ symbols in $\Sigma$ occurs
exactly $|R|$ times in $M'$. It then follows that $\nu$ cannot be
duplicated within row $x$ or column $y$, so $M'$ is a Latin square.
\end{proof}

We can now use \lref{l:nearrow} to prove the following result;
c.f.~\cite[Lemma $7$]{1subsq}.

\begin{lem}\label{l:nearnearcopy}
Let $L$ be a near copy of a Latin square $L'$ with associated alien
entry $\pi$. Let $T$ be a submatrix of $L$ that does not contain
$\pi$. Suppose that $T$ is a $k$-near copy of some $N_\infty$ square
$N$, where $k \in\{0,1,2\}$.  Suppose further that no symbol of an entry of $T$ that is alien with respect to $N$ is native to $T$.
Also suppose that $L$ has a subsquare $S$ that meets $T$ in at least
two entries. Let $V = S\cap T$, let $R$ be the set of rows of $V$ and let $C$ be the set of columns of $V$. Then,
\begin{itemize}
	\item If $k=0$ then $S$ contains $T$,
	\item If $k=1$ then let the alien entry in $T$ be $(r,c,\sigma)$.
	One of the following is true:
	\begin{enumerate}
		\item $S$ contains $T$,
		\item $R = \{r\}$ and $C = \{c, c'\}$ for some $c'$. Furthermore,
		$\pi$ is in column $c$ and has symbol $L[r, c']$,
		\item $R = \{r, r'\}$ for some $r'$ and $C = \{c\}$. Furthermore,
		$\pi$ is in row $r$ and has symbol $L[r', c]$,
		\item $R = \{r'\} \neq \{r\}$ and $C = \{c, c'\}$ for some
		$c'$. Furthermore, $\pi$ is in column $c'$, has symbol $L[r', c]$
		and $L[r', c']$ is the displaced native from the hole in $T$ in
		column $c$,
		\item $R = \{r, r'\}$ for some $r'$ and $C = \{c'\} \neq\{c\}$.
		Furthermore, $\pi$ is in row $r'$ and has symbol $L[r,c']$ and
		$L[r', c']$ is the displaced native from the hole in
		$T$ in row $r$.    
	\end{enumerate}
	\item If $k=2$ then one of the following is true:
	\begin{enumerate}
		\item $S$ contains $T$,
		\item $|R| \leq 3$ and $|C| \leq 3$ with $\min(|R|, |C|)<3$,
		\item $T$ has alien entries $(r, c, \sigma) \in V$ and
		  $(r',c',\sigma') \not\in V$.
                  The shadow of $V$ is a subsquare of the matrix
                  $\nu\inc N[r,c]$ where $\nu$ is the displaced native
                  from $(r',c')$. Also, either $r=r'$ and $\pi$ is in
                  column $c$, or $c=c'$ and $\pi$ is in row $r$.
	\end{enumerate} 
\end{itemize}
\end{lem}

\begin{proof}
If $k=0$ then the claim is true by \lref{l:nearrow}.
Suppose that $k=1$.
Since $V$ contains at least two entries we know that either
$|R|\geq2$ or $|C| \geq 2$. If both $|R| \geq 2$ and $|C|\geq 2$ then
\lref{l:nearrow} implies that $S$ contains $T$.  We consider only the
case where $|R|=1$ and $|C| \geq 2$. The case where $|C|=1$ and
$|R|\geq2$ can be resolved by transposing our arguments.  First
suppose that $V$ contains the alien entry of $T$ with respect to
$N$. So we can write $R = \{r\}$ and we know that $c \in C$. Let
$c'\in C \setminus \{c\}$ and let $\nu = L[r, c']$. Since $T$ contains
only one hole it follows that $\nu$ does occur in column $c$ of $T$,
say in row $r'$. Since $S$ is a Latin square that contains column $c$
and symbol $\nu$ but not row $r'$ it follows that $\pi$ occurs in
column $c$ and has symbol $\nu$. Furthermore, our choice of $c'$ was
forced and hence $C = \{c, c'\}$.

Now we consider when $V$ does not contain the alien entry of $T$ with
respect to $N$.
First suppose that $R = \{r\}$ so that $c \not\in C$. 
Let $c_1$ and $c_2$ be distinct elements of $C$, and for $i \in [2]$ let $\nu_i = L[r, c_i]$. Without loss of generality $\pi$ does not occur in column $c_1$. Since $c \neq c_1$ it follows that $\nu_2$ occurs in column $c_1$ of $T$, say in row $r_1$. But $S$ is a Latin square that contains column $c_1$, symbol $\nu_2$ but not row $r_1$, which is a contradiction. This contradiction implies that $R \neq \{r\}$. Now consider when $R = \{r'\} \neq \{r\}$ and $|C| \geq 2$. Assuming that $T\not\subseteq S$, \lref{l:nearrow} implies that $|C|=2$ and so we can write $C = \{c_1,c_2\}$. For $i \in [2]$ let $\nu_i = L[r', c_i]$. Without loss of generality $c_2 \neq c$. We know that $\nu_1$ occurs in column $c_2$ of $T$, say in row $r_2$. Since $S$ contains column $c_2$ and symbol $\nu_1$ but not row $r_2$ it follows that $\pi$ occurs in column $c_2$ and has symbol $\nu_1$. If $c_1 \neq c$, or $c_1=c$ and $\nu_2$ is not the displaced native from the hole in $T$ in column $c$ then the same argument we just applied to $\nu_1$ can be applied to $\nu_2$ to show that $\pi$ occurs in column $c_1$, which is false. Thus $c_1=c$ and $\nu_2$ is the displaced native from the hole in $T$ in column $c$. 

Finally, we deal with the $k=2$ case.
By \lref{l:nearrow} it suffices to consider the cases when $|R|\le2$
and $|C|\ge4$ or when $|R|\ge4$ and $|C|\le2$. We treat the former
case; the latter case can be resolved by transposing our arguments.
Suppose that $c_1,c_2,c_3,c_4$ are distinct columns in $C$ and let
$r_1 \in R$. By relabelling if necessary, we may assume that
$\{c_1,c_2\}\cap\{c, c'\}=\emptyset$ and $\pi$ does not occur in
column $c_1$. Let $\nu = L[r_1,c_2]$. Since $\nu$
is native to $T$ it follows that $\nu$ occurs in column $c_1$ of
$T$, say in row $r_2$. Since $S$ is a Latin square that contains
column $c_1$ and symbol $\nu$ and $\pi$ does not occur in column
$c_1$ it follows that $S$ must contain row $r_2$.
Now, since $k=2$ there must be at least three symbols in $V$ that are
native to $T$. Each of them must occur in column $c_1$ of $V$, contradicting
that $|R|\le2$.
\end{proof}

The following lemma is straightforward.

\begin{lem}\label{l:3cyc}
Let $L$ be a Latin square with a row cycle $\rho$ of length $3$. Suppose that $S$ is a subsquare of $L$ that contains more than one entry in $\rho$. Then $S$ contains all entries of $\rho$.
\end{lem}

We will need the following two results, which are analogous to
\lref{l:3cyc} in the case where $L$ is a near copy of a Latin square.

\begin{lem}\label{l:near3cyc}
Let $L$ be a near copy of a Latin square with alien entry
$\pi$. Suppose that $L$ contains the entries,
\[
\mathcal{D} = \{(r_1,c_1,k_1), (r_2,c_1,k_2), (r_1,c_2,k_2), (r_2,c_2,k_3), (r_1,c_3,k_3), (r_2,c_3,k_1)\},
\]
for rows $r_1$ and $r_2$, columns $c_1$, $c_2$ and $c_3$, and symbols $k_1$, $k_2$ and $k_3$. Suppose that $\pi \not\in \mathcal{D}$. Also suppose that $S$ is a subsquare of $L$ that contains entry $(r_1,c_1,k_1)$. Then one of the following holds:
\begin{itemize}
	\item $S \cap \mathcal{D} = \mathcal{D}$, or
	\item $S \cap \mathcal{D}$ contains at most one entry in $\{(r_1,c_2,k_2), (r_1,c_3,k_3)\}$ and none of the entries in $\{(r_2,c_1,k_2), (r_2,c_2,k_3), (r_2,c_3,k_1)\}$.
\end{itemize}
\end{lem}

\begin{proof}
Let $R$ be the set of rows of $S$. We will first show that if $r_2 \in
R$ then $S$ contains $\mathcal{D}$. If $r_2 \in R$ then $S$ contains
the entry $(r_2,c_1,k_2)$. So $S$ is a Latin square that contains
row $r_1$ and symbol $k_2$. It follows that $S$ must contain column
$c_2$ or $\pi$ occurs in row $r_1$ of $L$ and has symbol
$k_2$. Suppose first that $S$ does not contain column $c_2$. Since $S$
is a Latin square that contains row $r_2$ and symbol $k_1$ it follows
that $S$ must contain column $c_3$ because $\pi$ is in row $r_1$. 
Therefore $S$ also contains symbol
$k_3$. Since $S$ contains symbol $k_3$ and row $r_2$ it follows that
$S$ must also contain column $c_2$ because $\pi$ is in row
$r_1$. Hence $S$ contains $\mathcal{D}$. Now suppose that $S$ does
contain column $c_2$. Then $S$ is a Latin square that contains row
$r_2$ and symbol $k_1$ and so $S$ contains column $c_3$ unless $\pi$
is in row $r_2$ and has symbol $k_1$. Similarly, since $S$ contains
symbol $k_3$ and row $r_1$ we know that $S$ must contain column $c_3$
unless $\pi$ is in row $r_1$ and has symbol $k_3$. Therefore $S$ must
contain column $c_3$, hence $S$ contains $\mathcal{D}$.

We now consider the case where $r_2 \not\in R$. If $(r_1,c_2,k_2)\in S$
then since $S$ is a Latin square that contains
symbol $k_2$ and column $c_1$ but does not contain row $r_2$, we know
that $\pi$ must be in column $c_1$ and have symbol $k_2$. Similarly if
$(r_1,c_3,k_3) \in S$ then $\pi$ must be in column
$c_3$ and have symbol $k_1$. The lemma follows because $\pi$ cannot be
in both column $c_1$ and $c_3$.
\end{proof}

\begin{lem}\label{l:near3cycalt}
Let $L$ be a near copy of a Latin square with alien entry
$\pi=(r_1,c_1,k_1)$ with displaced native $k_4$.
Suppose that $L$ contains the entries,
\[
\mathcal{D} = \{(r_1,c_1,k_1), (r_2,c_1,k_2), (r_1,c_2,k_2), (r_2,c_2,k_3), (r_1,c_3,k_3), (r_2,c_3,k_4)\},
\] 
for rows $r_1$ and $r_2$, columns $c_1$, $c_2$ and $c_3$, and
distinct symbols $k_1$, $k_2$, $k_3$ and $k_4$. Suppose that $S$ is
a subsquare of $L$ that contains $\pi$. Then
$S \cap \mathcal{D} \subseteq \{(r_1,c_1,k_1), (r_1,c_3,k_3)\}$.
\end{lem}

\begin{proof} 
Let $R$ be the set of rows of $S$. We first show that $r_2 \notin R$.
If $r_2 \in R$ then $S$
contains the entry $(r_2,c_1,k_2)$. Since $S$ is a Latin square
that contains symbol $k_2$ and row $r_1$ it follows that $S$ must
also contain column $c_2$. Hence $S$ also contains symbol $k_3$ and
therefore $S$ must also contain column $c_3$ and hence also symbol
$k_4$. However, $k_4$ does not occur in row $r_1$, which contradicts
the fact that $S$ is a subsquare.

If $S$ contains $(r_1,c_2,k_2)$
then since $S$ is a Latin square that contains column $c_1$ and
symbol $k_2$ it follows that $S$ must also contain row $r_2$, which
we have just shown is impossible.
\end{proof}

As mentioned in \sref{s:back}, we will utilise corrupted products in our construction of $N_\infty$ Latin squares. So we will need some corrupting pairs. Throughout the rest of the paper we will be using the following four Latin squares frequently, and the symbols $A_8$, $B_8$, $A_9$ and $B_9$ will be reserved for them.
\begin{align}\label{e:AB8}
A_8 =	\left[\begin{array}{cccccccc}
	4&8&6&7&5&1&3&2\\
	8&6&4&2&7&5&1&3\\
	1&7&5&3&4&2&6&8\\
	5&4&3&1&2&6&8&7\\
	3&2&1&\mk4&6&\mk8&\mk7&5\\
	2&1&7&5&8&3&4&6\\
	6&3&2&\mk8&1&\mk7&\mk5&4\\
	7&5&8&6&3&4&2&1\\
\end{array}\right]
\quad
B_8 = \left[\begin{array}{cccccccc}
	4&1&7&2&8&6&5&3\\
	7&3&5&8&6&1&4&2\\
	3&5&8&4&1&7&2&6\\
	2&7&4&6&3&8&1&5\\
	1&8&6&5&4&2&3&7\\
	6&4&3&7&2&5&8&1\\
	5&2&1&3&7&4&6&8\\
	8&6&2&1&5&3&7&4\\
\end{array}\right]
\end{align}
\begin{align}\label{e:AB9}
A_9 =\left[\begin{array}{ccccccccc}
	2&8&6&3&1&4&5&9&7\\
	8&6&2&9&5&1&3&7&4\\
	3&\mk4&7&1&\mk2&5&6&8&\mk9\\
	1&3&5&2&4&9&7&6&8\\
	9&1&8&7&3&2&4&5&6\\
	7&\mk2&1&6&\mk9&3&8&4&\mk5\\
	4&5&9&8&7&6&1&2&3\\
	5&7&3&4&6&8&9&1&2\\
	6&9&4&5&8&7&2&3&1\\
\end{array}\right]
\quad
B_9 = \left[\begin{array}{ccccccccc}
	2&4&3&7&8&6&9&5&1\\
	3&7&9&5&4&8&2&1&6\\
	4&6&1&3&9&2&5&7&8\\
	6&2&4&9&5&1&8&3&7\\
	7&5&6&8&1&4&3&9&2\\
	5&9&2&1&3&7&6&8&4\\
	1&3&8&2&6&5&7&4&9\\
	8&1&5&6&7&9&4&2&3\\
	9&8&7&4&2&3&1&6&5\\
\end{array}\right]
\end{align}

Let $\alpha \in \{8,9\}$, and let $A = A_\alpha$ and $B = B_\alpha$. The following properties of $A$ and $B$ can be verified computationally.

\begin{enumerate}[Property 1:]
\item $(A, B)$ is a corrupting pair.
\item For $i \in [3]$ let $d_i = A[1,i]$. The row permutation $\tau_{1,2}$ of $A$ contains the cycle $(d_1,d_2,d_3)$. Furthermore, $\{d_i+1 : i \in [3]\} \cap \{d_i : i \in [3]\} = \emptyset$. 
\item 
As highlighted in \eref{e:AB8} and \eref{e:AB9}, there is a row permutation $\tau_{i, j}$ of $A$ with $3 \leq i < j$, and symbol $k$ such that $\tau_{i, j}^3(k) = k+1 \not\in \{d_1,d_2,d_3\}$ and none of $k$, $\tau_{i, j}(k)$ or $\tau_{i, j}^2(k)$ occur in cell $(i, 1)$.
\item $\{((i, j), (i', j')) \in ([2] \times [3])^2 : A[i, j] = B[i', j']\} = \{((1,1), (1,1))\}$.
\item The only matrix in the set,
\begin{equation}\label{e:asub}
	\{d_1 \inc A[2,1], d_2 \inc A[1,1], d_2 \inc A[2,2], d_3 \inc A[1,2], d_3 \inc A[2,3], d_1 \inc A[1,3]\}
\end{equation}
that contains a subsquare of order at least two is the matrix $d_1 \inc A[1,3]$. Furthermore, any proper subsquare of this matrix is an intercalate.

\item Suppose that $C$ is one of the matrices in \eref{e:asub}, that
$D$ is a matrix in
\begin{equation}\label{e:a2sub}
	\{B[i, j] \inc C[i, j] : \{i, j\} \subset [\alpha]\},
\end{equation}
and that $S$ is a square submatrix of $D$ that includes two alien entries
with respect to $A$. Then $S$ contains at least two different symbols.
Also, if $S$ is a subsquare then it is an intercalate. If $S$
is an intercalate that includes the principal entry of $D$, then its two alien
entries with respect to $A$ both occur within the first row of $D$ or both
occur within the first column of $D$.

\item No matrix in the set 
\begin{equation}\label{e:prop7}
	\begin{aligned}
		\{&d_2\inc A[1,1], d_3\inc A[1,1], d_1\inc A[1,2], 
		d_1\inc A[1,3], d_1\inc A[2,1]\}
	\end{aligned}
\end{equation}
contains a subsquare of order more than two.

\end{enumerate}

The definitions of $d_1,d_2,d_3$ from Property $2$ will be fixed for
the remainder of this section. Also, there is overlap between Property
$5$ and Property $7$, but it is convenient to state them both given
the distinct roles that these properties will play in our proof.

Let $\X$ denote the set of pairs $(L, s)$ where $L$ is an $N_\infty$ Latin square of order $\mu \geq 10$ with row indices, column indices and symbol set $[\mu]$, and $s \in [\mu-1]$ such that the following conditions hold:
\begin{enumerate}[Condition X]
\item[Condition $(i)$:] $s$ is a strong allowable shift with respect to $(A_8,L)$ and $(A_9,L)$.
\item[Condition $(ii)$:] $L$ contains a row cycle of length $3$ that involves rows $x_1$, $x_2$, columns $y_1$, $y_2$, $y_3$ and symbols $z_1$, $z_2$, $z_3$ with $1 \not\in \{x_1,x_2,y_1,y_2,y_3\}$ and $L[1,1]+s \not\in \{z_1,z_2,z_3\}$. Moreover, the matrix $L[x_2,y_3] \inc L[x_1,y_3]$ does not contain an intercalate.

\item[Condition $(iii)$:] There exist rows $r_1$, $r_2$, columns $c_1$, $c_2$, $c_3$ and a symbol $\sigma$ of $L$, with $1 \not\in \{r_1,r_2,c_1,c_2,c_3\}$, such that $L$ contains the entries,
\begin{equation}\label{e:proto3cyc}
	\{(r_1,c_1,\sigma), (r_2,c_1,\tau(\sigma)), (r_1,c_2,\tau(\sigma)),(r_2,c_2,\tau^2(\sigma)), (r_1,c_3,\tau^2(\sigma)), (r_2,c_3,\sigma+s)\}
\end{equation}
where $\tau=\tau_{r_1,r_2}$. Also, the matrix $L[r_2,c_3] \inc L[r_1,c_3]$ contains no intercalates and neither $L[1,1]$ nor $L[1,1]+s$ are elements of $\{\sigma, \tau(\sigma), \tau^2(\sigma), \sigma+s\}$.
\end{enumerate}

Define $N(\X) = \{\mu \in \mathbb{Z} : \text{there is a pair } (L,
s)\in \X \text{ where $L$ is of order $\mu$}\}$.  We aim to show that
$N(\X)$ contains all integers of the form $2^x3^y\ge10$. We will do
this using a recursive construction involving corrupted products.
Condition $(iii)$ together with Property $2$
is used to ensure that we will have a row cycle of length $3$
available to switch to destroy $\beta_M$, the unique proper subsquare
in the corrupted product. The difference between the symbols in the
first and last entries in \eref{e:proto3cyc} accounts for the shift
by $s$ that occurs when $\beta_M$ is created.
Great care is needed to ensure that we do not create new subsquares
in our recursive step. Several of the properties of $(A,B)$ have been
designed with this in mind. Also, Condition $(iii)$
includes subconditions to ensure we do not create intercalates.
Condition $(ii)$ is needed for the recursive step, in order to
ensure that Condition $(iii)$ can be satisfied for the subsequent
step. 

\subsection{Base cases}\label{ss:basecases}

In this subsection we create suitable base cases for our recursion.
We will show that 
\begin{equation}\label{e:base}
\{12,16,18,24,32,36,48,54,64,72\}
\subseteq N(\X). 
\end{equation}
In every instance we will use $s=1$.
Suppose that some pair $(L,1)$ satisfies Condition
$(iii)$ with some rows $r_1$, $r_2$, columns $c_1$, $c_2$, $c_3$, and
symbol $\sigma$. We will simply say that $L$ satisfies Condition
$(iii)$ with rows $r_1$, $r_2$ and symbol $\sigma$, since the
columns $c_1$, $c_2$ and $c_3$ are uniquely determined by this
information. Similarly, if $(L, 1)$ satisfies Condition $(ii)$ with row cycle $\rho(i, j, c)$ then we will simply say that $L$ satisfies Condition $(ii)$ with row cycle $\rho(i, j, c)$. Furthermore, we will always choose $i$, $j$ and $c$, respectively, to play the roles of $x_1$, $x_2$ and $y_3$ in Condition $(ii)$.
We will also not explicitly say that each of our base cases
satisfies Condition $(i)$; this is something that can easily be checked.

Let $L_{12}$ denote the $N_\infty$ Latin square of order $12$ constructed by Gibbons and Mendelsohn~\cite{ninf12}. Then $L_{12}$ satisfies Condition $(ii)$ with row cycle $\rho(2,11,11)$, and satisfies Condition $(iii)$ with rows $3$, $8$ and symbol $10$. 
Let $L_{18}$ denote the $N_\infty$ Latin square of order $18$ constructed by Elliott and Gibbons~\cite{siman}. Then $L_{18}$ satisfies Condition $(ii)$ with row cycle $\rho(4,5,9)$ and satisfies Condition $(iii)$ with rows $2$, $11$ and symbol $10$.

We next give a construction for Latin squares which will show that
$\{16,32,64\} \subseteq N(\X)$.  Let $n \geq 4$ be a positive
integer satisfying $\gcd(n,6)=2$ and let $J$ be a Latin square of
order $3$. Let $\mathcal{C}$ denote the Latin square on symbols
$[n-3]$ defined by $\mathcal{C}_{i, j} = (i+j) \bmod (n-3)$. For
$k\in\{-1,0,1\}$ define the following set of entries of $\mathcal{C}$,
\[
\Theta_k = \{(2j-3k,j,3j-3k): j \in [n-3]\}.
\]
It is simple to see that the sets $\Theta_{-1}$, $\Theta_0$ and $\Theta_1$ are pairwise disjoint. A Latin square $K = K(n,J)$ can then be defined as follows. If $(i, j, \ell)$ is an entry of $\mathcal{C}$ which is not contained in any set $\Theta_k$ then $K[i, j]=\ell$. If $(i, j, \ell) \in \Theta_k$ then $K[i, j]=n+2+k$ and $K[i, n+2+k] = K[n+2-k, j] = \ell$. Finally, $K[n-3+i, n-3+j]=n+J[i, j]$ for each $\{i, j\} \subseteq [3]$. For our purposes we will take
\begin{align*}
J=\left[\begin{array}{ccc}
	3 & 2 & 1 \\
	1 & 3 & 2 \\
	2 & 1 & 3
\end{array}\right].
\end{align*}
The squares $K(n,J)$ were first constructed by Kotzig and Turgeon~\cite{kotzturg}. In \cite{1subsq} it was shown that when $n-3$ is prime the only proper subsquare of $K(n,J)$ is the copy of $J$ in the bottom right corner.

Let $n \in \{16,32,64\}$ and let $L_n$ be the Latin square obtained from $K(n, J)$ by switching on $\eta(n, 1,n/2-1)$, then switching the resulting square on $\eta(4,n-1,n-2)$. Then $L_n$ satisfies Condition $(ii)$ with row cycle $\rho(2,8,n/2+4)$, and satisfies Condition $(iii)$ with rows $2$, $3$ and symbol $n-2$. 

To show that $\{24,36,48,54\} \subseteq N(\X)$ we will use a new construction. Let $e$ be a positive, even integer and let $E$ be an $N_\infty$ square of order $e$. Let $(k, \ell)$ be a cell in $E$ and let
\begin{align*}
Z=\left[\begin{array}{ccc}
	1 & 2 & 3 \\
	2 & 3 & 1 \\
	3 & 1 & 2
\end{array}\right].
\end{align*}
Define a Latin square $\mathcal{L} = \mathcal{L}(E, (k, \ell))$ by first forming a matrix whose rows and columns are indexed by $[3] \times [e]$, and where the cell $((i, j), (x, y))$ is occupied by
\[
\begin{cases}
(2,E[k, \ell]) & \text{if } (i, x) \in \{(1,1), (2,3), (3,2)\} \text{ and } (j, y) = (k, \ell), \\
(1,E[k, \ell]) & \text{if } (i, x) \in \{(1,2), (2,1), (3,3)\} \text{ and } (j, y) = (k, \ell), \\
(Z[i, x], E[j, y]+e/2) & \text{if } (i, x) = (2,2), \\
(Z[i, x], E[j, y]) & \text{otherwise}.
\end{cases}
\]
We then rename the rows indices, column indices and symbols using $\prec_1$
to obtain $\mathcal{L}$. Intuitively, we are obtaining $\mathcal{L}$ from the direct product $Z \times E$ as follows. Add $e/2$ to the $E$-coordinate of each symbol in the $E$-block of $Z \times E$ in position $(2,2)$. Then switch the resulting square on a symbol cycle of length $3$ between the symbols $(1,E[k, \ell])$ and $(2,E[k, \ell])$.

Henceforth we fix $E$ to be the $N_\infty$ square in \eref{e:E}.
Let $L_{24}$ denote the Latin square obtained from $\mathcal{L}(E, (1,2))$ by switching on $\eta(6,14,18)$. Then $L_{24}$ satisfies Condition $(ii)$ with row cycle $\rho(2,7,16)$ and satisfies Condition $(iii)$ with rows $2$, $5$ and symbol $3$. Let $L_{36}$ denote the Latin square obtained from $\mathcal{L}(L_{12}, (2,3))$ by switching on $\eta(1,21,30)$. Then $L_{36}$ satisfies Condition $(ii)$ with row cycle $\rho(2,11,35)$ and satisfies Condition $(iii)$ with rows $2$, $3$ and symbol $15$. Let $L_{48}$ denote the Latin square obtained from $\mathcal{L}(L_{16}, (16,8))$ by switching on $\eta(1,17,41)$. Then $L_{48}$ satisfies Condition $(ii)$ with row cycle $\rho(2,8,12)$ and satisfies Condition $(iii)$ with rows $2$, $3$ and symbol $14$. 
Let $L_{54}$ denote the Latin square obtained from $\mathcal{L}(L_{18}, (7,7))$ by switching on $\eta(1,19,52)$. Then $L_{54}$ satisfies Condition $(ii)$ with row cycle $\rho(2,10,29)$ and satisfies Condition $(iii)$ with rows $2$, $8$ and symbol $34$. 

Finally, to show that $72 \in N(\X)$ consider the corrupted product $P = (A_9,B_9) *_5 E$. Let $L_{72}$ be obtained from $P$ by renaming the row indices, column indices and symbols using $\prec_1$ and then switching on $\rho(2,11,3)$.
Then $L_{72}$ satisfies Condition $(ii)$ with row cycle $\rho(2,7,48)$ and satisfies Condition $(iii)$ with rows $2$, $5$ and symbol $19$.
We have shown \eref{e:base}.

\subsection{The recursive step}

We will prove the following theorem, by combining Lemmas \ref{l:qninf},
\ref{l:Qmshift} and \ref{l:Qii} below.

\begin{thm}\label{t:X}
If $\mu \in N(\X)$ then $\{8\mu, 9\mu\} \subseteq N(\X)$.
\end{thm}

The following notation will be fixed throughout this subsection.
Let $A = A_\alpha$ and $B = B_\alpha$, where $\alpha \in \{8,9\}$. Let
$(M, s) \in \X$ with $M$ of order $\mu$ and let $P = (A, B) *_s
M$. Since $(M, s) \in \X$ we know that $M$ satisfies Condition $(iii)$
with some rows $r_1$, $r_2$, columns $c_1$, $c_2$, $c_3$, and symbol
$\sigma$.
By definition of $P$ we have that $P[(1,r_1), (1,c_1)] = (d_1,\sigma+s)$.
Combining this with Condition $(iii)$ and Property $2$,
we see that the row cycle $\rho((1,r_1), (2,r_2), (3,c_3))$
of $P$ has length $3$.  Let $Q$ be obtained from $P$ by switching on
this row cycle.  We will use $\tau$ to denote the row permutation
$\tau_{r_1,r_2}$ of $M$. Then $Q$ contains the row cycle
$\rho((1,r_1), (2,r_2), (3,c_3))$.  Denote this row cycle by
$\mathcal{D}$. Define $Q_1$ to be the Latin square obtained from $Q$
by using $\prec_1$ to relabel the row indices, column indices and
symbols of $Q$ to be the set $[\alpha\mu]$. Formally, we relabel by using the
map $\varphi : [\alpha] \times [\mu] \to [\alpha\mu]$ defined by
$\varphi(i, j) =\mu(i-1)+j$.
In order to prove \tref{t:X}, we will show that $(Q_1,\mu) \in \X$.
Note that if $\varphi(i, j) = k$ then $\varphi(i+1,j) = k+\mu$.
For each cell $(\gamma, \delta)$ of $Q$ we define $Q'$ by
$Q' = (Q[\gamma, \delta]+(1,0)) \inc Q[\gamma, \delta]$.
Hence to show that
$\mu$ is a strong allowable shift with respect to $Q_1$ it suffices to
show for every cell $(\gamma, \delta)$ of $Q$ that $Q'$ contains no proper
subsquare of order $\alpha$.

Notice that $Q$ is a $6$-near copy of $P$ and $Q'$ is just a relabelling of $Q$.
The following lemma exhibits strong restrictions on the intersection between
a subsquare and a block in such matrices.

\begin{lem}\label{l:blocksP}
Let $P'$ be a $\ell$-near copy of $P$ for some non-negative integer
$\ell \leq 7$. Suppose that $P'$ is a near copy of a Latin square
$P''$ with associated alien entry $\pi$. Suppose that $\mathcal{D}\cup\{\pi\}$
contains all the alien entries of $P'$ with respect to $P$.
Let $T$ be a block of $P'$ in some position $(u, v)$ and assume that $\pi \not\in T$. Let $S$ be a subsquare of $P'$ containing $\pi$.  Suppose that $S$ contains more than one entry from $T$. Let $V = S \cap T$, let $R$ be the set of rows of $V$ and let $C$ be the set of columns of $V$.
Then,
\begin{enumerate}[$(i)$]
	\item If $T$ is an $M$-block that contains no alien entries with respect to $P$ then $T$ is not the principal $M$-block and one of the following is true: 
	\begin{enumerate}
		\item $R = \{(u, 1)\}$, $C = \{(v, 1), (v, c)\}$ for some $c \in [\mu]$, $\pi$ is in column $(v, 1)$ and has symbol $P'[(u, 1), (v, c)]$,
		\item $R = \{(u, 1), (u, r)\}$ for some $r \in [\mu]$, $C = \{(v, 1)\}$, $\pi$ is in row $(u, 1)$ and has symbol $P'[(u, r), (v, 1)]$,
		\item $R = \{(u, r)\}$ for some $r \in [\mu] \setminus \{1\}$, $C = \{(v, 1), (v, c)\}$ for some $c \in [\mu]$, $\pi$ occurs in column $(v, c)$ and has symbol $P'[(u, r), (v, 1)]$. Furthermore, $P'[(u, r), (v, c)] = (A[u, v], M[1,1])$,
		\item $R = \{(u, 1), (u, r)\}$ for some $r \in [\mu]$, $C = \{(v, c)\}$ for some $c \in [\mu] \setminus \{1\}$, $\pi$ occurs in row $(u, r)$ and has symbol $P'[(u, 1), (v, c)]$. Furthermore, $P'[(u, r), (v, c)] = (A[u, v], M[1,1])$.
	\end{enumerate}
	\item If $T$ is an $A$-block that contains no alien entries with respect to $P$ then one of the following is true: 
	\begin{enumerate}
		\item $R = \{(1,u)\}$, $C = \{(1,v), (c, v)\}$ for some $c \in [\alpha]$, $\pi$ is in column $(1,v)$ and has symbol $P'[(1,u), (c, v)]$,
		\item $R = \{(1,u), (r, u)\}$ for some $r \in [\alpha]$, $C = \{(1,v)\}$, $\pi$ is in row $(1,u)$ and has symbol $P'[(r, u), (1,v)]$,
		\item $R = \{(r, u)\}$ for some $r \in [\alpha] \setminus \{1\}$, $C = \{(1,v), (c, v)\}$ for some $c \in [\alpha]$, $\pi$ occurs in column $(c, v)$ and has symbol $P'[(r, u), (1,v)]$. Furthermore, $P'[(r, u), (c, v)] = (A[1,1], M[u, v])$,
		\item $R = \{(1,u), (r, u)\}$ for some $r \in [\alpha]$, $C = \{(c, v)\}$ for some $c \in [\alpha] \setminus \{1\}$, $\pi$ occurs in row $(r, u)$ and has symbol $P'[(1,u), (c, v)]$. Furthermore, $P'[(r, u), (c,v)] = (A[1,1], M[u, v])$.
	\end{enumerate}
	\item If $T$ contains an alien entry with respect to $P$ then
	$|R|\leq 3$ and $|C| \leq 3$. Furthermore, $\min(|R|, |C|)<3$.
\end{enumerate}
\end{lem}

\begin{proof}
As $T$ is a block of $P'$ not containing $\pi$, by construction we
know that $\Phi_j(T)$ is a $k$-near copy of $N$ for some
$j\in\{1,2\}$, $k\in\{0,1,2\}$ and $N\in\{A,B,M,\beta_M\}$.  We will
first prove that $S$ cannot contain $T$.  We will prove this claim for
the case where $T$ is an $M$-block. The proof when $T$ is an $A$-block
is similar. Suppose, for a contradiction, that $S$ contains $T$.
Since $\pi\in S\setminus T$, we know that $S \neq T$.
Hence $S$ contains a whole row of another $M$-block, say
$U$, and a whole column of another $M$-block, say
$U'$. \lref{l:nearnearcopy} implies that $S$ contains $U$ also, unless
$\pi \in U$.
If $\pi \in U$ then \lref{l:nearnearcopy} implies that
$S$ contains $U'$. Either way, $S$ contains two $M$-blocks of $P'$,
which we will call $T$ and $T'$.

Let $\mathcal{A}$ be the set of $A$-blocks of $P'$ that do not contain
an alien entry of $P'$ with respect to $P$ or $P''$.  Each $A$-block
in $\mathcal{A}$ hits both $T$ and $T'$. Suppose that $\pi$ is
in row $(x, x')$ and column $(y, y')$. \lref{l:nearnearcopy} implies that $S$
contains every $A$-block in $\mathcal{A}$ which row $(x, x')$ and
column $(y, y')$ do not intersect. Since $\mu\ge10$ and
$\mathcal{D}\cup\{\pi\}$
contains all the alien entries of $P'$ with respect to $P$ or $P''$, we
know that for each $i \in [\mu] \setminus \{x'\}$ there is some
$k\in [\mu]$ such that $S$ contains the $A$-block of $P'$ in position $(i,k)$.
It follows that $R$ contains all rows not of the form $(r, x')$ for
some $r \in [\alpha]$. Hence the order of $S$ must be at least
$(\mu-1)/\mu$ times the order of $P'$. Since $\mu\ge10$, this is a
contradiction of \lref{l:nolargesubsq}, proving that $S$ does not contain
$T$.

Suppose that $T$ contains no alien entry with respect to $P$.  Then
the only possible hole in $\Phi_j(T)$ with respect to $N$ is the principal
entry.
By \lref{l:nearnearcopy} we infer that $k=1$, and
that either part $(i)$ or part $(ii)$ of the present Lemma holds.

Hence, it suffices to consider the case when part $3$ of the
$k=2$ case of \lref{l:nearnearcopy} occurs and
$T$ contains an alien entry with respect to $P$.  
If $T$ is an
$M$-block then this alien entry does not occur in the same row or
column as the principal entry of $T$, since
$1 \not\in \{r_1,r_2,c_1,c_2,c_3\}$.
Hence $T$ must be an $A$-block, and the alien entry
of $T$ with respect to $P$ must be in a cell in the set
$\{((1,r_1),(2,c_2)), ((1,r_1), (3,c_3)), ((2,r_2), (1,c_1))\}$.
\lref{l:nearnearcopy} implies that the shadow of $V$
is a subsquare in one of the matrices in \eref{e:prop7}. Property $7$ then
implies that $|R|=|C|=2$, completing the proof.
\end{proof}

An idea that we will use repeatedly is to consider the projection of a
hypothetical subsquare onto one of the factors in a corrupted product.
Our next lemma shows one circumstance where we know this projection is
a Latin square.

\begin{lem}\label{l:projLS}
Let $F$ be any matrix whose row indices, column indices and symbol
set are $[\alpha]\times[\mu]$ and let $S$ be a subsquare of $F$ of
order $t$. Suppose, for some $i \in [2]$, that the projection
$\Phi_i$ is injective on the rows and columns of $S$ and that there
is some Latin square $L$ such that $\Phi_i(S)$ agrees with its
shadow in $L$ in all but $\ell < t$ entries. Then $\Phi_i(S)$ is a
Latin square.
\end{lem}

\begin{proof}
Since the projection of $S$ onto $L$ is injective it
follows that $\Phi_i(S)$ is a $t \times t$ matrix. As $S$ is a Latin
square it follows that every symbol in $\Phi_i(S)$ occurs some
multiple of $t$ times. Since $\Phi_i(S)$ agrees with its shadow in
$L$ in all but $\ell<t$ places it follows that $\Phi_i$ is injective
on the set of symbols that occur in $S$ and hence $\Phi_i(S)$ is a
Latin square.
\end{proof}

Our next task is to show that $Q$ is $N_\infty$, which we do with the
following three lemmas.

\begin{lem}\label{l:1rho}
Any proper subsquare of $Q$ must contain exactly one entry from $\mathcal{D}$.
\end{lem}

\begin{proof}
Let $S$ be a proper subsquare of $Q$. Let $R$ be the set of rows of
$S$ and let $C$ be the set of columns of $S$. First suppose that $S$
does not contain an entry from $\mathcal{D}$. Then $P[R, C]$ is a
proper subsquare of $P$, which can only be $\beta_M$ by
\tref{t:1sub}. But $P[R, C] \neq \beta_M$ as otherwise $S$ would
contain an entry from $\mathcal{D}$. Now suppose that $S$ contains at
least two entries from $\mathcal{D}$. Then \lref{l:3cyc} implies that
$S$ contains every element from $\mathcal{D}$. It follows that
$P[R,C]$ is a proper subsquare of $P$ that contains all entries from
$\rho((1,r_1), (2,r_2), (3,c_3))$, but $P$ has no such subsquare.
\end{proof} 

\begin{lem}\label{l:q2}
Any proper subsquare of $Q$ hits every $M$-block of $Q$ at most once.
Also, any proper subsquare of $Q$ hits the principal $A$-block of $Q$
at most once.
\end{lem}

\begin{proof}
Let $S$ be a proper subsquare of $Q$. Let $R$ be the set of rows of $S$ and let $C$ be the set of columns of $S$. By \lref{l:1rho} we know that $S$ contains exactly one entry from $\mathcal{D}$. Let this entry be in cell $((i', j'), (x, y))$ for some $i' \in [2]$, $x \in [3]$, $\{j', y\} \subseteq [\mu] \setminus \{1\}$. So $Q[(i', j'), (x, y)] = P[(i, j), (x, y)]$ for $i \in [2] \setminus \{i'\}$ and some $j \in [\mu] \setminus \{1\}$. Let $P'$ denote the matrix $P[(i, j), (x, y)] \inc P[(i', j'), (x, y)]$ and let $\pi = ((i', j'), (x, y), P[(i, j), (x, y)])$ be the alien entry of $P'$ with respect to $P$. Then $S' = P'[R, C]$ is a proper subsquare of $P'$. We will first show that $S'$ contains at most one entry from every $M$-block of $P'$. 
Denote the $M$-block of $P'$ that contains $\pi$ by $T$. Let $T' \neq T$
be an $M$-block of $P'$, so that $T'$ has no alien entries with
respect to $P$. Let the position of $T'$ be $(u, v)$ and suppose that
$S'$ contains at least two entries from $T'$.
Since $1 \not\in \{j', y\}$, we know case $(i)(a)$ and $(i)(b)$
of \lref{l:blocksP} do not arise. So we may assume that case $(i)(c)$
occurs (the argument for case $(i)(d)$ is similar).

Hence we are assuming that $S' \cap T'$ has only row $(u, r)$ and columns $(v, 1)$ and $(v, c)$ for some $\{r, c\} \subseteq [\mu] \setminus \{1\}$ and $(v, c) = (x, y)$. Let $\nu_1=P[(i, j), (x, y)] = P'[(u, r), (v, 1)]$ be the symbol
in $\pi$ and $\nu_2=P'[(u, r), (v, c)] = (A[u, v], M[1,1])$.
Since $\Phi_1(\nu_1)=A[u, x]$, it follows that $u=i\in[2]$.
Let $u' \in [\alpha]$ be such that $B[u', v] = A[u, v]$. Note that Property $4$ implies that $u'>2$. Since $S'$ must contain symbol $\nu_2$ in column $(v, 1)$ it follows that $S'$ contains the entry $((u', 1), (v, 1), (B[u', v], M[1,1]))$. Let $T''$ be the $M$-block of $P'$ in position $(u', v)$. Since $\pi$ occurs in an $M$-block whose position is in the set $[2] \times [3]$ it follows that $\pi \not\in T''$.
Now running the above argument with $T''$ in place of $T'$,
we obtain the contradiction that $u'\in[2]$.

Next we show that $|S'\cap T|=1$. The position of $T$ is $(i',x)$. Suppose that $|S'\cap T|\ge2$ and $S'$ contains two rows that hit $T$. If $S'$ has a column $(c, c')$ where $c \neq x$, then $S'$ contains more than one entry from the $M$-block of $P'$ in position $(i', c)$, which is false. Hence every column of $S'$ hits $T$. Then by similar reasoning, every row of $S'$ hits $T$,
meaning that $S'$ is contained within $T$. We reach the same conclusion if
we start with an assumption that $S'$ contains two columns that hit $T$.
So $S'\subseteq T$.
All symbols in $T$ other than those in $\pi$ and the principal entry have $A$-coordinate $A[i', x]$. However, the symbol in $\pi$ has $A$-coordinate $A[i,x] \neq A[i',x]$, and it must occur in every row of $S'$. The only way this might happen is if the principal entry of $T$ has the same symbol as $\pi$. But that would require that $B[i', x] = A[i, x]$, which contradicts Property $4$. Hence
$S'$ must intersect every $M$-block of $P'$ at most once.

Finally, we suppose that $S'$ contains more than one entry from the principal $A$-block of $P'$. \lref{l:blocksP} implies that $\pi$ occurs in row $(r,1)$ for some $r \in [\alpha]$, or in column $(c,1)$ for some $c \in [\alpha]$. However, this contradicts that $1\notin\{j', y\}$.
\end{proof}

\begin{lem}\label{l:qninf}
The Latin square $Q$ is $N_\infty$.
\end{lem}

\begin{proof}
Suppose that $S$ is a proper subsquare of $Q$. Let $R$ be the set of rows of $S$ and let $C$ be the set of columns of $S$. \lref{l:q2} tells us that $S$ intersects every $M$-block of $Q$ at most once, and hits the principal $A$-block of $Q$ at most once. \lref{l:1rho} says that $S$ contains exactly one entry from $\mathcal{D}$. Let this entry be $\pi = ((i', j'), (x, y), P[(i, j), (x, y)])$ for some $\{i, i'\} = [2]$, $x \in [3]$ and $\{j, j', y\} \subseteq [\mu] \setminus \{1\}$. Denote the matrix $P[(i, j), (x, y)] \inc P[(i', j'), (x, y)]$ by $P'$. Then $S' = P'[R, C]$ is a proper subsquare of $P'$ that hits every $M$-block of $P'$ at most once. It follows that $\Phi_1$ is injective on $R$ and $C$.
Let the order of $S'$ be $t$. So $\Phi_1(S')$ is a $t \times t$ matrix that agrees with its shadow in $A$ in all but $v \in \{1,2\}$ entries.
Indeed, $v=1$ unless $S'$ contains the principal entry of some $M$-block of $P'$ other than the principal $M$-block.
\lref{l:projLS} implies that $\Phi_1(S')$ is a Latin square unless $t=v=2$. If $t=v=2$ then $\Phi_1(S')$ is either an intercalate or contains only one symbol, since each symbol in it occurs a multiple of $t$ times.

Suppose that $v=1$. Then $\Phi_1(S')$ is a subsquare of one of the matrices in the set \eref{e:asub}. Property $5$ implies that $S'$ is an intercalate, $i'=1$, $i=2$, $x=3$, $j' = r_1$, $j=r_2$ and $y = c_3$. We can write $R = \{(1,r_1), (r, r')\}$ and $C = \{(3,c_3), (c, c')\}$ for some $\{r, c\} \subseteq [\alpha]$ and $\{r', c'\} \subseteq [\mu]$. Since the symbols in cells $(1,1)$ and $(1,3)$ of the matrix $d_1 \inc A[1,3]$ agree, we know that $\Phi_1(S')$ does not contain the entry in cell $(1,1)$. It follows that $c\ne1$.
As $S'$ is an intercalate, we must have $M[r_2,c_3] = M[r', c']$ and $M[r_1,c'] = M[r', c_3]$. Since $r_1 \neq r_2$ it follows that the matrix $M[r_2,c_3] \inc M[r_1,c_3]$ must contain an intercalate, which contradicts Condition $(iii)$.

Thus we may assume that $v=2$ and $\Phi_1(S')$ is a Latin square, or a
$2 \times 2$ matrix with only one symbol. Also $\Phi_1(S')$, which is a
submatrix of one of the matrices in the set \eref{e:a2sub},
contains two alien entries with respect to $A$. Property $6$ then
tells us that $\Phi_1(S')$ is an intercalate, and hence so is $S'$.
Additionally, from $1\notin\{y,j'\}$ we know that the two alien
entries of $S'$ occur in different rows and columns. From Property $6$
and the injectivity of $\Phi_1$ on $S'$, we deduce
that $S'$ does not hit the principal $M$ block. Since $S'$ is an
intercalate, we need the symbols in its two alien entries to
match. But that requires $M[1,1]\in\{M[j,y],M[j,y]+s\}$, in
contradiction of Condition $(iii)$.  Thus $S$ cannot exist.
\end{proof}

We now work on showing for every cell $(\gamma, \delta)$ that $Q'$ has
no subsquare of order $\alpha$. Recall that
$Q'=(Q[\gamma,\delta]+(1,0)) \inc Q[\gamma, \delta]$.  For the rest of
this section let $\pi$ be the alien entry of $Q'$ with respect to $Q$
and let $\mathcal{D}'$ denote the set of entries of $Q'$ that occupy a
cell of some entry in $\mathcal{D}$. By definition,
$\mathcal{D}'\subseteq(\mathcal{D}\cup\{\pi\})$.  On several occasions
we will use that, by \lref{l:qninf}, any proper subsquare of $Q'$ must
contain $\pi$, and that $\Phi_2$ cannot distinguish between $\pi$ and
the corresponding entry in $Q$.

\begin{lem}\label{l:not3W}
Let $W$ be a block of $Q'$ containing $\pi$. Then no subsquare of $Q'$
of order $\alpha$ hits $W$ in at least three rows or in at least three
columns.
\end{lem}

\begin{proof}
Let $S$ be a subsquare of $Q'$ of order $\alpha$.
Suppose that $S$ hits $W$ in
at least three columns. We will first show that $S$ must be contained within $W$. We will give the argument assuming that $W$ is an $M$-block. Analogous arguments can be used to prove the same claim for $A$-blocks. Let the position of $W$ be $(u, v)$. Suppose that there is a row $(r, r')$ of $S$ which does not hit $W$. Let $T$ be the $M$-block of $Q'$ in position $(r, v)$. Then $S$ hits $T$ in three distinct columns. \lref{l:blocksP} implies that $T$ must contain an alien entry with respect to $P$, hence $(r, v) \in [2] \times [3]$. It also implies that $S$ hits $T$ in at most two rows and exactly three columns. Hence $S$ can contain at most four rows which do not hit $W$. Since $\alpha \geq 8$ there are at least four rows of $S$ which hit $W$. However only three columns of $S$ hit $W$ and therefore there is a column $(c, c')$ of $S$ which does not hit $W$. Then $S$ hits the $M$-block of $Q'$ in position $(u, c)$ in at least four rows, contradicting \lref{l:blocksP}. This contradiction implies that all rows of $S$ must hit $W$. But then $S$ must be contained within $W$, as otherwise $S$ would hit an $M$-block of $Q'$ in at least eight rows, again contradicting \lref{l:blocksP}. 

Now suppose that $S$ hits $W$ in at least three rows, and is not contained within $W$. First we claim that $u \in [2]$. If not, then since at most two columns of $S$ hit $W$ it follows that $S$ has a column $(c, c')$ that does not hit $W$. The $M$-block of $Q'$ in position $(u, c)$ has no alien entries with respect to $P$ and $S$ hits this $M$-block in three rows, which contradicts \lref{l:blocksP}. So $u \in [2]$. If $S$ has a column $(c, c')$ with $c>3$ that does not hit $W$ then again $S$ must hit an $M$-block of $Q'$ that has no alien entries with respect to $P$ in at least three rows. Hence every column of $S$ either hits $W$ or is of the form $(c, c')$ with $c \in [3]$. \lref{l:blocksP}, combined with the fact that at least three rows hit $W$, implies that for each $i \in [3]$, $S$ has at most two columns of the form $(i, i')$. The only possibility is that $\alpha=8$, $v > 3$, exactly two columns of $S$ hit $W$ and for each $i \in [3]$, $S$ has exactly two columns of the form $(i, i')$. In that case \lref{l:blocksP} implies that $S$ has at most three rows of the form $(1,u')$ and at most three rows of the form $(2,u'')$. It follows that there is a row $(r, r')$ of $S$ with $r>2$. Let $T$ be the $M$-block of $Q'$ in position $(r, 1)$. We know that $S$ hits $T$ in two distinct columns and so \lref{l:blocksP} implies that $v=1$, which is false because $v>3$. This contradiction implies that $S$ must be contained within $W$.

We now show that $S$ cannot be contained within $W$. Suppose that $W$ is an $M$-block. Let $\Omega$ be the set of entries
$e$ in $S$ that satisfy (i)~$e=\pi$,
(ii)~$e$ is the principal entry of $W$ and/or (iii)~$e \in \mathcal{D}'$. 
Since $S$ is a Latin square it
follows that each symbol in $\Phi_1(S)$ occurs some multiple of $\alpha$
times. Also, all symbols in $W\setminus\Omega$ have the same
$A$-coordinate.  Since $\alpha>3$, it follows that every symbol in $W$ must
have the same $A$-coordinate. Now consider an entry $e$ that satisfies (iii).
If $e$ is an entry of $\mathcal{D}$ then its symbol has the wrong
$A$-coordinate to match with symbols in $W\setminus\Omega$.
The only possible fix would be if $e=\pi$,  but that is ruled out by 
Property 2. So no entry in $\Omega$ satisfies (iii).
Hence any entry in $\Omega$ has the same $M$-coordinate as the
corresponding entry in $P$.
If $W$ is not the principal $M$-block of $Q'$ then 
$\Phi_2(S)$ is an $\alpha \times \alpha$ Latin square that agrees with its
shadow in the matrix $M$. However $M$ is $N_\infty$ of order $\mu > \alpha$,
so this is impossible. We reach a similar contradiction
if $W$ is the principal $M$-block of $Q'$, because
then $\Phi_2(S)$ is an $\alpha \times \alpha$ Latin square that agrees with its
shadow in the matrix $M'$ defined by $M'[i, j] = M[i, j]+s$ for all
$\{i, j\} \subseteq [\mu]$.

We now consider the case when $W$ is an $A$-block of $Q'$.
Note that $S=W$ because they have the same order. Define $\Omega$ as
we did above. Then arguing similarly, we find that every symbol in $W$
has the same $M$-coordinate. Therefore the entry of $W$ that satisfies
(ii) must also satisfy (iii). Also, \lref{l:projLS} implies that $\Phi_1(S)$ is a Latin square, and from \lref{l:mintrade} it 
agrees with $B$ if $W$ is the principal $A$-block of $Q'$ and agrees
with $A$ otherwise. Hence the $A$-coordinate of the symbol in the principal
entry of $W$ must be equal to $A[1,1]$. This necessitates that the
entry satisfying (ii) and (iii) is $\pi$, but this contradicts
Property $2$.
\end{proof}

\begin{lem}\label{l:q'2}
Any subsquare of $Q'$ of order $\alpha$ hits every block of $Q'$ at most
once.
\end{lem}

\begin{proof}
We give the argument for $M$-blocks; the argument for $A$-blocks is
similar.  Let $W$ be the $M$-block containing $\pi = ((x_1,x_2),
(y_1,y_2), z) \in ([\alpha] \times [\mu])^3$.  By \lref{l:blocksP} and
\lref{l:not3W} we know that $S$ does not hit any block in three
rows. Hence there are at least two rows $(x_1,x'_1)$ and
$(x_2,x'_2)$ of $S$ that do not hit $W$ and satisfy $2<x_1\le x_2$.

Suppose that $S$ hits some $M$-block in two columns, say $(c, c')$
and $(c, c'')$, where $\pi$ is not in $(c,c')$.
Now \lref{l:blocksP} implies that $x_1<x_2$ and 
$Q[(x_1,x_1'),(c,c')]=z=Q[(x_2,x_2'),(c,c')]$, which violates the fact that
$Q$ is a Latin square.

We thus know that $S$ contains two columns $(y_1,y'_1)$ and
$(y_2,y'_2)$ that do not hit $W$ and satisfy $3<y_1<y_2$.
Employing a similar argument, we can deduce that no $M$-block
can be hit by two rows of $S$.
\end{proof}

\begin{lem}\label{l:Qmshift}
The integer $\mu$ is a strong allowable shift with respect to $(A, Q_1)$.
\end{lem}

\begin{proof}
It suffices to show that $Q'$ has no subsquare of order $\alpha$.
Suppose that $S$ is a subsquare of $Q'$ of order $\alpha$ that
contains $u$ elements from $\mathcal{D}'$. By \lref{l:near3cyc} and
\lref{l:near3cycalt} we know that $u\in \{0,1,2,6\}$ and that if
$u=6$ then $\pi\notin\mathcal{D}'$.  Let $R$ be the set of rows of
$S$ and let $C$ be the set of columns of $S$.  By \lref{l:q'2}, we
know that $S$ hits every block of $Q'$ at most once and so $\Phi_1$
and $\Phi_2$ are injective on $R$ and $C$.

First suppose that $u \in \{0,6\}$.  Let $P' = (P[\gamma,
\delta]+(1,0)) \inc P[\gamma, \delta]$ and let $S' = P'[R,
C]$. Then $S'$ is a subsquare of $P'$ that hits every block of
$P'$ at most once.  Since the $M$-coordinate of the symbol of $\pi$
is equal to $\Phi_2(P[\gamma,\delta])$, it follows that $\Phi_2(S')$
is a matrix of order $\alpha$ that agrees with its shadow in $M$ in
all but at most one entry. \lref{l:projLS} implies that $\Phi_2(S')$
is a Latin square. If $S'$ contains an entry from the principal
$M$-block of $P'$ then $\Phi_2(S')$ is a subsquare of the matrix
$(M[i, j]+s) \inc M[i, j]$ for some $\{i, j\} \subseteq [\mu]$,
which contradicts the fact that $s$ is a strong allowable shift with
respect to $(A, M)$. If $S'$ does not hit the principal $M$-block of
$P'$ then $\Phi_2(S')$ is a proper subsquare of $M$, which is a
contradiction because $M$ is $N_\infty$.

It remains to consider the case when $u \in \{1,2\}$.  Define
$\Omega$ to be the set of entries $e$ in $S$ that satisfy (i) $e=\pi$,
(ii) $e$ is the principal entry of some $M$-block of $Q'$ and/or (iii)
$e \in\mathcal{D}'$. By \lref{l:q'2} and $u\le2$, we know that
$|\Omega| \leq 4$. Since $\Phi_1(S)$ is a matrix of order $\alpha$
that agrees with its shadow in $A$ except possibly on the entries coming from
$\Omega$, \lref{l:projLS} implies that $\Phi_1(S)$ is a Latin square.
Furthermore, by \lref{l:mintrade} and the fact that $A$ is $N_\infty$
it follows that $\Phi_1(S) = A$. Since $u\ge1$, there exists an entry
$e$ of $S$ that satisfies (iii). Note that $e$ cannot satisfy (ii),
because $1\notin\{r_1,r_2\}$.  If $e\in \mathcal{D}$ then $\Phi_1(e)$
does not match its shadow in $A$.  So we must have $e=\pi$, but then
$\Phi_1(e)$ still does not match its shadow in $A$ by Property $2$, a
contradiction.
\end{proof}

\lref{l:Qmshift} takes care of Condition $(i)$ in our recursive step.
Next we deal with the other two required conditions, and thereby complete
the proof of \tref{t:X}.

\begin{lem}\label{l:Qii}
The pair $(Q_1,\mu)$ satisfies Condition $(ii)$ and Condition $(iii)$.
\end{lem}

\begin{proof}
We know that $Q$ contains the row cycle $\mathcal{D} = \rho((1,r_1), (2,r_2), (3,c_3))$ and $r_1 \neq 1$. Also, the columns involved in $\mathcal{D}$ are $(1,c_1)$, $(2,c_2)$ and $(3,c_3)$ with $c_1 \neq 1$. Switching $Q$ on $\mathcal{D}$ yields $P$ which does not contain an intercalate by \tref{t:1sub}. Hence the matrix $Q[(2,r_2), (3,c_3)] \inc Q[(1,r_1), (3,c_3)]$ does not contain an intercalate either. Also, since $M$ satisfies Condition $(iii)$,
$$\Phi_2\big(Q[(1,1), (1,1)]+(1,0)\big)=M[1,1]+s\not\in\big\{\sigma+s, \tau(\sigma), \tau^2(\sigma)\big\}.$$
So $Q[(1,1), (1,1)]+(1,0)$ is not in $\{(d_1,\sigma+s), (d_2,\tau(\sigma)), (d_3,\tau^2(\sigma))\}$, which is the set of symbols of $\mathcal{D}$. 
It follows that $(Q_1,\mu)$ satisfies Condition $(ii)$ with row cycle $\rho(\varphi(1,r_1), \varphi(2,r_2), \varphi(3,c_3))$

Since $(M, s)\in\X$ we know that $M$ satisfies Condition $(ii)$ with some row cycle involving rows $x_1$, $x_2$, columns $y_1$, $y_2$, $y_3$ and symbols $z_1,z_2,z_3$. Without loss of generality $z_f = M[x_1,y_f]$ for each $f \in [3]$.
Also, we know that $A$ satisfies Property $3$ with some rows $\{i, j\}$ with $3 \leq i < j$, columns $\{\ell_1,\ell_2,\ell_3\}$ and symbol $k$ such that $\tau_{i, j}^f(k) = A[j, \ell_f]$ for each $f \in [3]$. It follows that $Q$ contains the entries,
\begin{equation}\label{e:newnear3cyc}
	\begin{aligned}
		&\big((i, x_1), (\ell_1,y_1), (k, z_1)\big),
		&\big((i, x_1), (\ell_2,y_2), (\tau_{i, j}(k), z_2)\big), 
		&\quad\big((i, x_1), (\ell_3,y_3), (\tau_{i, j}^2(k), z_3)\big),\\
		&\big((j, x_2), (\ell_1,y_1), (\tau_{i, j}(k), z_2)\big), 
		&\big((j, x_2), (\ell_2,y_2), (\tau_{i, j}^2(k), z_3)\big),
		&\quad\big((j, x_2), (\ell_3,y_3), (k+1,z_1)\big).
	\end{aligned}
\end{equation}
Next we show that the matrix $Q[(j, x_2), (\ell_3,y_3)] \inc Q[(i, x_1), (\ell_3,y_3)]$ contains no intercalates. Suppose that $I$ is an intercalate of this matrix with rows $\{(i, x_1), (r, r')\}$ and columns $\{(\ell_3,y_3), (c, c')\}$ for some $\{r, c\} \subseteq [\alpha]$ and $\{r', c'\} \subseteq [\mu]$. Property $3$ tells us that $k+1 \notin \{d_1,d_2,d_3\}$ which implies that $(r, c) \neq (1,1)$ and $I$ contains no entry from $\mathcal{D}$. It follows that $M[x_2,y_3] = M[r', c']$ and $M[x_1,c'] = M[r', y_3]$. Since $x_1 \neq x_2$ this is a contradiction of Condition $(ii)$. By Condition $(ii)$ we also know that $M[1,1]+s \not\in \{z_1,z_2,z_3\}$. Thus neither $Q[(1,1), (1,1)]$ nor $Q[(1,1), (1,1)]+(1,0)$ are elements of $\{(k, z_1), (\tau_{i, j}(k), z_2), (\tau_{i, j}^2(k), z_3), (k+1,z_1)\}$.
Therefore $(Q_1,\mu)$ satisfies Condition $(iii)$ with rows $\varphi(i, x_1)$, $\varphi(j, x_2)$, columns $\varphi(\ell_1,y_1)$, $\varphi(\ell_2,y_2)$, $\varphi(\ell_3,y_3)$ and symbol $\varphi(k, z_1)$, as illustrated in \eref{e:newnear3cyc}. Note that $1\notin\{x_1,x_2,y_1,y_2,y_3\}$ ensures that
\[1\notin\{\varphi(i, x_1),\,\varphi(j, x_2),\,\varphi(\ell_1,y_1),
\,\varphi(\ell_2,y_2),\,\varphi(\ell_3,y_3)\}.\qedhere
\]
\end{proof}

We are now ready to prove our main result for this section.

\begin{proof}[Proof of \tref{t:ninf}]
By prior results it suffices to show that there exists an $N_\infty$ Latin square of order $n$ for all $n \geq 12$ of the form $2^x3^y$, where $x \geq 1$ and $y \geq 0$. Let $n$ be such an integer. Write $n = 2^{3i+j}3^{2k+\ell}$ for some $\{i, j, k, \ell\} \subseteq \mathbb{Z}$ with $j \in \{0,1,2\}$ and $\ell \in \{0,1\}$. Consider the following table. 

\def\tspacer{{\vrule height 4.5ex width 0ex depth0ex}}	
\begin{center}
	\begin{tabular}{|c|c|c|}
		\hline
		& $\ell=0$ & $\ell=1$ \\
		\hline
		\tspacer
		$j=0$ & \begin{tabular}{cc}
			$i \geq 2$ & $n=8^{i-2}9^k64$ \\
			$i=1$ and $k \geq 1$ & $n=8^{i-1}9^{k-1}72$ 
		\end{tabular}
		& $n=8^{i-1}9^k24$ \\
		\hline
		\tspacer                        
		$j=1$ & \begin{tabular}{cc}
			$i \geq 1$ & $n=8^{i-1}9^k16$ \\
			$i=0$ and $k \geq 1$ & $n=9^{k-1}18$ 
		\end{tabular}&
		\begin{tabular}{cc}
			$i \geq 1$ & $n=8^{i-1}9^k48$ \\
			$i=0$ and $k \geq 1$ & $n=9^{k-1}54$ 
		\end{tabular} \\
		\hline
		\tspacer                        
		$j=2$ & \begin{tabular}{cc}
			$i \geq 1$ & $n=8^{i-1}9^k32$ \\
			$i=0$ and $k \geq 1$ & $n=9^{k-1}36$ 
		\end{tabular} & $n=8^i9^k12$ \\
		\hline
	\end{tabular}
\end{center}
This table, together with the base cases in \eref{e:base}, tells us that we can always write $n$ in the form $8^i9^kn'$ for some $n' \in N(\X)$. We can then repeatedly apply \tref{t:X} to show that $n \in N(\X)$. 
\end{proof}

\section{Latin hypercubes without subhypercubes}\label{s:main2}

In this section we prove \tref{t:ninfhyper}. Let $H$ be a $d$-dimensional Latin hypercube of order $n$. Let $d' \geq d$ be an integer. We can define an array $\mathcal{H}_{d'}(H) : [n]^{d'} \to [n]$ by,
\[
\mathcal{H}_{d'}(H)[x_1,x_2,\ldots, x_{d'}] \equiv H[x_1,x_2,\ldots, x_d] +\sum_{i=d+1}^{d'} x_i \bmod n.
\]
The following lemma is easy to verify.

\begin{lem}
Let $H : [n]^d \to [n]$ be a $d$-dimensional Latin hypercube of order $n$ and let $d' \geq d$ be an integer. Then $\mathcal{H}_{d'}(H)$ is a Latin hypercube.
\end{lem}

Moreover, boosting the dimension in this way preserves the $N_\infty$ property:

\begin{lem}\label{l:hninf}
Let $H : [n]^d \to [n]$ be a $d$-dimensional $N_\infty$ Latin hypercube of order $n$ and let $d' \geq d$ be an integer. The Latin hypercube $\mathcal{H}_{d'}(H)$ is $N_\infty$.
\end{lem}

\begin{proof}
Suppose, for a contradiction, that $\mathcal{H}_{d'}(H)$ has a proper subhypercube $S = H|_{S_1 \times S_2 \times \cdots \times S_{d'}}$ for some subsets $S_i \subseteq [n]$. Let the symbol set of $S$ be $\Xi \subseteq [n]$ of cardinality $k$. For each $i \in [d'] \setminus [d]$ let $s_i \in S_i$ and let $s = \sum_{i=d+1}^{d'} s_i$. Define $H' : S_1 \times \cdots \times S_d \to (\Xi+s)$ by $H'[x_1,x_2,\ldots, x_d] = S[x_1,x_2,\ldots, x_d, s_{d+1}, \ldots, s_{d'}] = H[x_1,x_2,\ldots, x_d]+s \bmod n$. 
It is easy to see that $H'$ is a subhypercube of the Latin hypercube $H'' : [n]^d \to [n]$ defined by $H''[x_1,x_2,\ldots, x_d] \equiv H[x_1,x_2,\ldots, x_d] +s \bmod n$. This implies that $H$ has a proper subhypercube, which is a contradiction.
\end{proof}

Let $H : [n]^3 \to [n]$ be a Latin cube of order $n$. For each $x \in [n]$ the restriction $H|_{[n] \times [n] \times \{x\}}$ induces a Latin square $L_x$ defined by $L_x[i, j] = H[i, j, x]$. We can specify $H$ by listing the Latin squares $L_x$ in order for each $x \in [n]$. McKay and Wanless~\cite{hyper} enumerated Latin hypercubes of small orders and some of the data can be found in~\cite{mckay}. The Latin cube specified by \eref{e:43} is an $N_\infty$ Latin cube of order four. There are five species of Latin cubes of order four and only one of them is $N_\infty$. The Latin cube specified by \eref{e:63} is also $N_\infty$ and has order six. There are $264248$ species of Latin cubes of order six and $17946$ of them are $N_\infty$.
\begin{align}\label{e:43}
\left[\begin{array}{cccc}
	1&2&3&4\\
	2&3&4&1\\
	3&4&1&2\\
	4&1&2&3\\
\end{array}\right|
\begin{array}{cccc}
	2&1&4&3\\
	1&4&3&2\\
	4&3&2&1\\
	3&2&1&4\\
\end{array}
\left|
\begin{array}{cccc}
	3&4&2&1\\
	4&2&1&3\\
	2&1&3&4\\
	1&3&4&2\\
\end{array}\right|
\left.\begin{array}{cccc}
	4&3&1&2\\
	3&1&2&4\\
	1&2&4&3\\
	2&4&3&1\\
\end{array}\right]
\end{align}

\begin{equation}\label{e:63}
\begin{gathered}
	\left[\begin{array}{cccccc}
		1&2&3&4&5&6\\
		2&1&4&3&6&5\\
		3&4&6&5&1&2\\
		4&3&5&6&2&1\\
		5&6&1&2&4&3\\
		6&5&2&1&3&4\\
	\end{array}\right|
	\begin{array}{cccccc}
		2&1&4&3&6&5\\
		1&3&2&6&5&4\\
		6&2&5&4&3&1\\
		3&4&6&5&1&2\\
		4&5&3&1&2&6\\
		5&6&1&2&4&3\\
	\end{array}
	\left|
	\begin{array}{cccccc}
		3&4&5&6&1&2\\
		4&2&6&5&3&1\\
		5&6&3&1&2&4\\
		6&5&1&2&4&3\\
		1&3&2&4&6&5\\
		2&1&4&3&5&6\\
	\end{array}\right|\cdots \\
	\begin{array}{cccccc}
		4&3&6&5&2&1\\
		3&5&1&2&4&6\\
		2&1&4&6&5&3\\
		5&6&2&1&3&4\\
		6&4&5&3&1&2\\
		1&2&3&4&6&5\\
	\end{array}
	\left|\begin{array}{cccccc}
		5&6&1&2&3&4\\
		6&4&5&1&2&3\\
		1&5&2&3&4&6\\
		2&1&3&4&6&5\\
		3&2&4&6&5&1\\
		4&3&6&5&1&2\\
	\end{array}\right|
	\left.
	\begin{array}{cccccc}
		6&5&2&1&4&3\\
		5&6&3&4&1&2\\
		4&3&1&2&6&5\\
		1&2&4&3&5&6\\
		2&1&6&5&3&4\\
		3&4&5&6&2&1\\
	\end{array}\right]
\end{gathered}
\end{equation}

\tref{t:ninfhyper} now follows by combining \tref{t:ninf}, \eref{e:43} and \eref{e:63} with \lref{l:hninf}.

\section{Conclusion}\label{s:conc}

It seems likely that the method that we used to construct $N_\infty$
squares of orders of the form $2^x3^y$ could be generalised to
construct $N_\infty$ squares of many other orders.
We have used corrupting pairs of order 8 and 9 in our recursive step
that takes an $N_\infty$ Latin square of order $\mu$ and creates an
$N_\infty$ Latin square of order $8\mu$ or $9\mu$. There are no
corrupting pairs of order less than $7$, but they probably exist for
all orders 7 and above \cite{1subsq}.  Several of our arguments in the
proof of \tref{t:X} used that $\alpha \geq 8$. While those arguments
could not be directly applied if using a corrupting pair of order $7$,
we do not believe that there is an intrinsic obstacle to using such a
pair. 

Finally, we remark that now that the existence question is settled
for $N_\infty$ Latin squares, the next challenge is to find
asymptotic estimates of their number, along the lines of the estimates
for $N_2$ Latin squares found in \cite{KSS22,KSSS22,KS18}.

\printbibliography

\end{document}